\documentclass[11pt,reqno]{article}
\usepackage[margin=1in]{geometry}
\usepackage{dsfont, amssymb,amsmath,amscd,latexsym, amsthm, amsxtra,amsfonts}
\usepackage{lineno}
\usepackage[all]{xy}
\usepackage[active]{srcltx}

\usepackage{tikz}
\usepackage{float}
\usepackage[round]{natbib}
\usepackage{bbm}
\usepackage{enumerate}
\usepackage{mathrsfs}
\usepackage{graphicx}
\usepackage{comment}
\usepackage{mathtools}
\usepackage{cases}
 \usepackage{tikz}
\usetikzlibrary{calc,arrows}
\usepackage{verbatim}
\usepackage{graphicx}
\usepackage{subfigure}
\usepackage{color}

\usepackage{tikz}
\usetikzlibrary{calc,arrows}
\usepackage{verbatim}
\usepackage{graphicx}
\usepackage{subfigure}
\usepackage{epstopdf}
\usepackage{verbatim}
\usepackage{graphicx}
\usepackage{subfigure}
\usepackage{color}
\usepackage[affil-it]{authblk}

\newtheorem{theorem}{Theorem}[section]
\newtheorem{corollary}[theorem]{Corollary}
\newtheorem{proposition}[theorem]{Proposition}
\newtheorem{lemma}[theorem]{Lemma}

\theoremstyle{definition}
\newtheorem{definition}[theorem]{Definition}

\theoremstyle{definition}

\theoremstyle{definition}
\newtheorem{remark}{Remark}
\theoremstyle{definition}
\newtheorem{assumption}{Assumption}

\renewcommand{\epsilon}{\varepsilon}

\DeclareMathOperator{\var}{Var}

\usepackage[pdfstartview=FitH, bookmarksnumbered=true,bookmarksopen=true, colorlinks=true, citecolor=blue, linkcolor=blue,urlcolor=blue]{hyperref}
\usepackage{graphics}
\graphicspath{{figures/}}
\title{Local $W^{2,1}_{p}$ convergence of penalization for time-inconsistent stopping problems}

\author{Xiaodong Luo\thanks{Department of Applied Mathematics, The Hong Kong Polytechnic University, Kowloon, Hong Kong 
  (\url{xiaodong.luo@polyu.edu.hk}).}
\qquad Xiang Yu\thanks{Department of Applied Mathematics, The Hong Kong Polytechnic University, Kowloon, Hong Kong 
(\url{xiang.yu@polyu.edu.hk}).}	}
\numberwithin{equation}{section}

\usepackage{amsopn}

\date{\vspace{-0.5in}}
\begin{document}

\maketitle

% REQUIRED
\begin{abstract}
This paper studies the penalization method for a general class of time-inconsistent stopping problems. Inspired by some existing results on penalization for optimal stopping problems, we introduce the penalized PDE problem and the penalized control problem for time-inconsistent stopping problems, uncovering some interesting discrepancy stemming from the time-inconsistency. As the penalization tends to infinity, we develop some novel arguments to rigorously establish the local $W^{2,1}_{p}$ convergence results for the solution of the penalized PDE problem to a solution of the variational inequality associated to the time-inconsistent stopping problem. Under an additional concavity assumption, we further derive the convergence of the penalized time-inconsistent control problem to the
time-inconsistent stopping problem. We revisit two examples of time-inconsistent stopping to numerically illustrate the convergence theory.

\vskip 10 pt \noindent
{\bf Mathematics Subject Classification:} 60G40, 93E20, 35R35  
 \vskip 10pt  \noindent
{\bf Keywords:} time-inconsistent stopping, penalization method, local convergence, pure-strategy weak equilibrium, extended HJB variational inequality
\vskip 5pt\noindent
\end{abstract}

\section{Introduction}
Selecting the best time to stop and collect the rewards has always been a fundamental topic in modern decision theory. A wide range of financial applications fit into this type of optimization problems as stop-vs-continue, such as the derivative pricing (e.g., \cite{bensoussan1984theory}, \cite{karatzas1988pricing}, \cite{kim1990analytic}), asset trading (e.g., \cite{rosenfield1983optimal}, \cite{henderson2008explicit}, \cite{dutoit2009selling}), real options (e.g., \cite{brennan1985evaluating}, \cite{mcdonald1986value}, \cite{dixit1994investment}), among others. 

Based on dynamic programming principle or martingale optimality principle, the value function of an optimal stopping problem is typically characterized by the solution of the HJB variational inequality rather than a standard PDE, which gives rise to a free boundary problem with obstacles. To numerically solve such variational inequalities or theoretically prove the existence and the regularity of solutions, the penalization method has been developed as an important tool in the literature. That is, one introduces an additional parameter $K\uparrow+\infty$ or $\epsilon=\frac{1}{K}\downarrow0$ into a penalty term of the system. The penalization then effectively transforms the variational inequality problem into a more tractable nonlinear PDE problem. A standard form of the penalty term is $K(\cdot)^{+}$ (see \cite{scholz1986numerical}, \cite{forsyth2002quadratic}, and more recently \cite{witte2011penalty}, \cite{howison2013effect}, among others), which is not smooth at the origin but is both notationally and computationally simple. Such a form is convenient for numerical or explicit calculation of the solution to the penalized problem. Recent advances in deep learning (\cite{peng2024deep}) and reinforcement learning (\cite{dai2026learninga}) have shed new light on this form of penalization. Other variants of the penalization form are also widely studied in numerical computation as well as in existence theories for the variational inequalities; see \cite{brezis1974smoothness}, \cite{kikuchi1981convergence}, \cite{friedman1982variational}, \cite{zvan1998penalty}, to name a few.

An important mathematical problem related to the penalization method is its theoretical convergence as the penalization goes to infinity. In existing studies on the numerical method of variational inequalities, penalization is rarely analyzed in isolation; rather, it is inherently coupled with spatial and temporal discretizations. Within this framework, a natural question is the $L^{\infty}$ convergence over bounded domains (e.g., \cite{nochetto1989sharpl}, \cite{witte2011penalty}), as it guarantees pointwise accuracy that is essential for evaluating value functions and identifying optimal stopping boundaries. In terms of weak Sobolev convergence, standard results only yield estimates of the first-order derivatives bounded in the $L^2$-norm; see, e.g.,  \cite{kikuchi1981convergence}, \cite{scholz1986numerical}. In another stream of literature where the penalization is considered as a tool to study the existence and regularity of solutions, the weak convergence in Sobolev spaces can usually be strengthened to $W^{2,1}_{p}$. However, in order to establish estimates for weak derivatives, the penalty terms are usually only allowed to be formulated in smooth forms to facilitate the differentiation; see \cite{brezis1974smoothness}, \cite{friedman1975parabolic}, \cite{friedman1982variational}, among others. Recent progress in \cite{peng2024deep, dai2026learninga} for the purpose of machine learning has motivated the study of convergence for non-smooth penalization schemes of the form $K(\cdot)^{+}$ without discretization. However, the $L^{\infty}$ convergence established in these studies strictly relies on the boundedness of a specific expression. Because the boundedness is often violated in general setups with unbounded domains, their applicability remains restricted to a small class of problems such as American option pricing. Noting that state spaces in many financial  applications are naturally unbounded, the localized convergence analysis is a more suitable framework, which is one motivation of this study.

Another new feature of this study is time-inconsistency when the global optimality fails, that is, the optimal decision selected at the initial time may be rejected at later dates. Such phenomenon arises in problems under the non-exponential discounting, mean-variance criteria, etc. Some recent progress on time-inconsistent stopping problems are built on different concepts of pure-strategy equilibrium in continuous-time models that are described as hitting times of some stopping regions. The early work \cite{HuangNguyenHuu2018} proposed the pure-strategy mild equilibrium as the fixed point of an iteration operator and established its existence under decreasing impatient discount functions. See some further extensions in different context in \cite{huang2020general}, \cite{huang2020optimal}, \cite{HY21}, \cite{HW21}, \cite{DYZ25} and references therein. Later, some discussions on pure-strategy weak equilibrium based on perturbations of stopping policy in the intra-personal game was studied in \cite{christensen2018finding}, \cite{ebert2020weighted}, \cite{Bay21} and \cite{Bay23}, etc. The definition of weak equilibrium leads to the extended HJB variational inequality, but most previous studies only established its existence in specific examples such as stopping problems under some particular non-exponential discounting. A general sufficient condition for the existence of pure-strategy weak equilibria remains underdeveloped. To fill this gap, the present paper aims to provide a new sufficient condition to address the existence of pure-strategy weak equilibrium in a general class of time-inconsistent stopping problems using a penalization approach with the non-smooth penalization function $K(\cdot)^{+}$.

The main contributions of this work are summarized as below.

First, we mathematically formulate the penalized {\bf PDE problem} and the penalized {\bf control problem} for a general class of time-inconsistent stopping problems, and reveal a discrepancy in the time-inconsistent setting: the penalized {\bf PDE problem} may not coincide with the HJB system of the penalized {\bf control problem}. Notably, such gap does not appear in the classical time-consistent setting. We note that this distinction comes from the nonlinear expectation in the objective function, and we further find that a concavity condition can connect the solution of the penalized {\bf PDE problem} to the extended HJB system of the penalized {\bf control problem}.

Second, under some nearly minimal conditions, we establish the $W^{2,1}_{p,loc}$ convergence (as well as $W^{2,1}_{p,loc}$ subsequence convergence under weaker conditions) of the solution $(V^{K},f^{K},g^{K})$ of  the penalized {\bf PDE problem} to some $(V,f,g)$ solving the extended variational inequality associated to the time-inconsistent stopping problem. As our main result, Theorem \ref{main_conv2} only requires the pointwise convergence of $V^{K}$, uniform polynomial growth conditions, and mild regularity assumptions on the state dynamics and the free boundary. Notably, the partial pointwise convergence and the polynomial growth conditions are conceptually necessary: On one hand, time-inconsistent problems generally lack the monotonic evolving structure inherent in time-consistent settings, and the equilibria are not necessarily unique; On the other hand, solutions violating polynomial growth are non-unique and financially unreasonable. In contrast, the regularity conditions on the state dynamics and free boundary serve as technical assumptions, required solely to upgrade subsequence convergence to full sequence convergence. 

Third, we also discuss some important extensions of the main convergence theory: The $W^{2,1}_{p,loc}$ subsequence convergence only requires some uniform polynomial growth, some uniform local boundedness as well as some equicontinuity; The equilibrium for the penalized {\bf control problem} converges to a candidate equilibrium for the stopping problem under an additional concavity assumption. Moreover, our theory for time-inconsistent case naturally reduces to the time-consistent case, thus providing a local convergence criterion that supplements current $L^{\infty}$ convergence theory. The theory provides theoretical validity of the penalization approach to study the pure strategy equilibrium for time-inconsistent stopping problems

Finally, we apply the penalization convergence theory to two classic examples of time-inconsistent stopping: one problem under the non-exponential discounting and one problem under mean-variance objective. Instead of directly calculating the equilibrium for the stopping problem, we first calculate the equilibria for the penalized problem and verify the conditions of our theorem to obtain the stopping equilibrium. Notably, global $L^{\infty}$ convergence  is no longer  valid in these two classic examples, which illustrates the necessity for a local convergence theory.

The rest of the paper is organized as follows. Section \ref{Sect_model} reviews a general formulation of time-inconsistent stopping problems, and introduces the penalized {\bf PDE problem} as well as the penalized {\bf control problem}. Section \ref{Sect_conv} is devoted to establishing the main convergence results of the paper.  Section \ref{Sect_coro} presents some important extensions of the main results in Section \ref{Sect_conv}. Section \ref{Sect_example} revisits two classic time-inconsistent stopping problems, and shows the applicability of our local convergence theory with numerical illustrations.

\section{Problem Formulation and the Penalized Problems}
\label{Sect_model}

We first review a general formulation of the time-inconsistent stopping problem as in \cite{bjork2021time}, and introduce the penalized problems. 

\subsection{The time-inconsistent stopping problem}

Let $(\Omega,\mathcal{F},\{\mathcal{F}_{t}\}_{t\ge 0}, \mathbb{P})$ be a filtered probability space satisfying the usual conditions. $\{B_{t}\}_{t\ge 0}$ is a standard one-dimensional Brownian motion on $(\Omega,\mathcal{F},\{\mathcal{F}_{t}\}_{t\ge 0}, \mathbb{P})$.

The time-inconsistent stopping problem at $(x,t)\in Q:=\mathbb{R}\times[0,+\infty)$ is to maximize
\begin{equation}
\label{object_stopping}
J(x,t;\tau):=\mathbb{E}_{x,t}F(X_{\tau},\tau,x,t)+G(\mathbb{E}_{x,t}[X_{\tau}],x,t)
\end{equation}
over the set of all stopping times in $[0,+\infty)$. The underlying state process $\{X_{r}\}_{r\ge t}$ follows the SDE
\begin{equation*}
\begin{cases}
dX_{r}=\mu(X_{r},r)dt+\sigma(X_{r},r)dB_{r},\quad r\in[t,+\infty),\\
X_{t}=x,
\end{cases}
\end{equation*}
where $\mu(\cdot,\cdot)$ and $\sigma(\cdot,\cdot)$ are given deterministic functions on $Q$. The reward functions $(x,t,z,s)\mapsto F(x,t,z,s)$ and $(y,z,s)\mapsto G(y,z,s)$ are also given deterministic functions on $\mathbb{R}\times[0,+\infty)\times \mathbb{R}\times[0,+\infty)$ and $\mathbb{R}\times\mathbb{R}\times[0,+\infty)$.

The time-inconsistency of objective (\ref{object_stopping}) stems from the dependence of initial state-time as well as the non-linear functional of expectation. The model covers different sources of time-inconsistency; see the examples in Section \ref{Sect_example} and the general formulation in \cite{bjork2021time}.

\begin{remark}
The analysis of the paper also works if the objective (\ref{object_stopping}) has an additional running reward term $\mathbb{E}_{x,t}\int_{t}^{\tau}H(X_{r},r,x,t)dr$. However, in many financial applications and in much of the previous literature, the current setup is enough and widely adopted.
\end{remark}

We make the following basic assumptions:
\begin{assumption}
\label{ass_basic}
Assume:\\
(1) $\mu$ and $\sigma$ are Lipschitz continuous, $F\in C^{2,1,2,1}$ and $G\in C^{2,2,1}$.\\
(2) $\mu(x,t)$ and $\sigma(x,t)$ are of at most linear growth in $x$. For $\varphi=F,F_{t},F_{s},F_{x},F_{z},F_{xx},F_{xz},F_{zz}$, there exists $C$ and $M$ independent of $t,s$ such that $|\varphi(x,t,z,s)|\le C(|x|^{M}+|z|^{M}+1)$ for any $(x,t,z,s)$ with $t\ge s$. For $\varphi=G,G_{s},G_{y},G_{z},G_{yy},G_{yz},G_{zz}$, there exists $C$ and $M$ independent of $s$ such that $|\varphi(y,z,s)|\le C(|y|^{M}+|z|^{M}+1)$.
\end{assumption}

In this paper, we focus on the pure Markov stopping strategy. A pure stopping strategy can be presented by some measurable set $W\subset Q$ denoting the continuation region. The stopping time generated by strategy $W$ is denoted by $\tau^{W}=\inf\{r\ge t|(X_{r},r)\notin W\}$. The definition of equilibrium stopping time is as follows.
\begin{definition}[equilibrium stopping time]
Given a continuation region $W$ and the corresponding stopping time $\hat{\tau}:=\tau^{W}$, then $\hat{\tau}$ is an equilibrium stopping time if, for all $(x,t)\in Q$,
\begin{align*}
&J(x,t;\hat{\tau})\ge F(x,t,x,t)+G(x,x,t),\\
&\liminf\limits_{h\rightarrow0+}\frac{J(x,t;\hat{\tau})-J(x,t;\hat{\tau}\vee(t+h))}{h}\ge 0.
\end{align*}
\end{definition}

Under some technical assumptions, the equilibrium stopping time $\tau^{W}$ with the equilibrium continuation region $W$ is characterized by the following extended HJB system for $(V,\{f^{z,s}\}_{(z,s)\in Q},g)$:
\begin{equation}
\label{vi}
\begin{aligned}
&\max\{\mathcal{A}V(x,t)-\mathcal{A}f(x,t,x,t)+\mathcal{A}f^{x,t}(x,t)+(\mathcal{H}g)(x,t)\\&-\mathcal{A}(G\diamond g)(x,t),F(x,t,x,t)+G(x,x,t)-V(x,t)\}=0,\\
&\mathcal{A}f^{z,s}(x,t)=0,\forall (x,t)\in W,\\
&f^{z,s}(x,t)=F(x,t,z,s),\forall (x,t)\notin W,\\
&\mathcal{A}g(x,t)=0,\forall (x,t)\in W,\\
&g(x,t)=x,\forall (x,t)\notin W,
\end{aligned}
\end{equation}
where the equilibrium waiting region has the feedback form
\begin{equation*}
W=\{(t,x)|F(x,t,x,t)+G(x,x,t)-V(x,t)<0\},
\end{equation*}
and the operators are defined by
\begin{align*}
&\mathcal{A}\varphi(x,t):=\varphi_{t}(x,t)+\mu(x,t)\varphi_{x}(x,t)+\frac{1}{2}\sigma^{2}(x,t)\varphi_{xx}(x,t),\quad \forall \varphi,\\
&\mathcal{H}\varphi(x,t):=G_{y}(\varphi(x,t),x,t)\mathcal{A}\varphi(x,t),\quad\forall\varphi,\\
&(G\diamond \varphi)(x,t)=G(\varphi(x,t),x,t),\quad\forall\varphi.
\end{align*}
 $V(x,t)$ presents the equilibrium value function, and $f(x,t,z,s):=f^{z,s}(x,t), g(y,z,s)$ are the auxiliary functions with probabilistic interpretations
\begin{align*}
&V(x,t)=J(x,t;\tau^{W}),\\
&f^{z,s}(x,t)=\mathbb{E}_{x,t}F(X_{\tau^{W}},\tau^{W},z,s)\\
&g(x,t)=\mathbb{E}_{x,t}[X_{\tau^{W}}].
\end{align*}

The main focus of the paper is to tackle the problem (\ref{vi}) by developing the penalization method. Due to the page limit, we omit the technical conditions and the standard verification proof of (\ref{vi}), and refer the details to \cite{christensen2018finding} and \cite{bjork2021time}.

\subsection{The penalized PDE problem and the penalized control problem}

Inspired by the existing results in time-consistent optimal stopping, we introduce the following PDE system for $(V^{K},\{(f^{K})^{z,s}\}_{(z,s)\in Q},g^{K})$ as the penalized {\bf PDE problem} of (\ref{vi}):
\begin{equation}
\label{pen}
\begin{aligned}
&\mathcal{A}V^{K}(x,t)-\mathcal{A}f^{K}(x,t,x,t)+\mathcal{A}(f^{K})^{x,t}(x,t)+(\mathcal{H}g^{K})(x,t)-\mathcal{A}(G\diamond g^{K})(x,t)\\&+K(F(x,t,x,t)+G(x,x,t)-V^{K}(x,t))^{+}=0,\\
&\mathcal{A}(f^{K})^{z,s}(x,t)+K(F(x,t,z,s)-(f^{K})^{z,s}(x,t))1_{\{(x,t)\notin W^{K}\}}=0,\\
&\mathcal{A}g^{K}(x,t)+K(x-g^{K}(x,t))1_{\{(x,t)\notin W^{K}\}}=0,
\end{aligned}
\end{equation}
where
\begin{equation*}
W^{K}:=\{(x,t)|F(x,t,x,t)+G(x,x,t)-V^{K}(x,t)<0\}.
\end{equation*}

We also introduce a penalized {\bf control problem} of (\ref{object_stopping}), which is to maximize 
\begin{equation}
\label{object_penal}
I^{K}(x,t;u):=\mathbb{E}_{x,t}\bigg\{\int_{t}^{\infty}Ku_{r}e^{-K\int_{t}^{r}u_{r'}dr'}F(X_{r},r,x,t)dr\bigg\}+G\bigg(\mathbb{E}_{x,t}\bigg\{\int_{t}^{\infty}Ku_{r}e^{-K\int_{t}^{r}u_{r'}dr'}X_{r}dr\bigg\},x,t\bigg)
\end{equation}
over the admissible control set $\mathcal{U}_{[t,+\infty)}:=\{\{u_{r}\}_{r\in[t,T]}|u_{r}\in[0,1],\forall r\in[t,+\infty)\}$.

However, we remark that the penalized {\bf PDE problem} (\ref{pen}) characterizes the penalized {\bf control problem} (\ref{object_penal}) only when there is no non-linear term, i.e., $G\equiv0$. For example, the property holds for problems with non-exponential discount or initial state dependence, but does not hold for mean-variance maximization problems where $G_{yy}> 0$.

To see this point, we give the verification theorem of the penalized control problem (\ref{object_penal}).
\begin{theorem}
Given $(V^{K},\{(f^{K})^{z,s}\}_{(z,s)\in Q},g^{K})$, assume in addition to Assumption \ref{ass_basic} that:\\
(a) The function $(x,t)\mapsto f^{K}(x,t,x,t)$ and $V^{K}(\cdot,\cdot),(f^{K})^{z,s}(\cdot,\cdot),g^{K}(\cdot,\cdot)$ are all in $W^{2,1}_{p,loc}(Q)\bigcap C(Q)$.\\
(b) The following extended HJB system holds:
\begin{align}
&\begin{aligned}
&\mathcal{A}V^{K}(x,t)-\mathcal{A}f^{K}(x,t,x,t)+\mathcal{A}(f^{K})^{x,t}(x,t)+(\mathcal{H}g^{K})(x,t)-\mathcal{A}(G\diamond g^{K})(x,t)\\&+K\big(F(x,t,x,t)+[G(g^{K}(x,t),x,t)-G_{y}(g^{K}(x,t),x,t)(g^{K}(x,t)-x)]-V^{K}(x,t)\big)^{+}=0,\\
\end{aligned}\label{penv1}\\
&\mathcal{A}(f^{K})^{z,s}(x,t)+K[F(x,t,z,s)-f^{K}(x,t,z,s)]\hat{u}(x,t)=0,\label{penf1}\\
&\mathcal{A}g^{K}(x,t)+K[x-g^{K}(x,t)]\hat{u}(x,t)=0,\label{peng1}
\end{align}
where $\hat{u}(x,t):=1_{\{(x,t)\notin\tilde{W}^{K}\}}$ is the candidate equilibrium feedback control with 
\begin{equation*}
\tilde{W}^{K}:=\{(x,t)\in Q|F(x,t,x,t)+G(g^{K}(x,t),x,t)-G_{y}(g^{K}(x,t),x,t)(g^{K}(x,t)-x)-V^{K}(x,t)<0\}.
\end{equation*}
(c) The transversality condition holds that, for $\varphi=(f^{K})^{z,s},g^{K}$ and $(x,t)\mapsto V^{K}(x,t)-f^{K}(x,t,x,t)-(G\diamond g^{K})(x,t)$,
\begin{equation*}
\lim_{T\to+\infty} \mathbb{E}_{x,t}\big[e^{-K\int_{t}^{T}\hat{u}(X_{r'},r')dr'}\varphi(X_{T},T)\big] = 0.
\end{equation*}
(d) The integrability condition holds that, for $\varphi=F(\cdot,\cdot,z,s),I^{x}(\cdot,\cdot)$,
\begin{equation*}
\mathbb{E}_{x,t}\int_{t}^{\infty} K\hat{u}(X_{r},r)e^{-K\int_{t}^{r}\hat{u}(X_{r'},r')dr'}\big\vert{}\varphi(X_{r},r)\big\vert{}dr < \infty,
\end{equation*}
where $I^{x}$ denotes the function $(x,t)\mapsto x$.\\
(e) There exists $\bar{h}>0$ such that for any $(x,t)\in Q$ and $h\in(0,\bar{h})$, 
\begin{equation*}
\mathbb{E}_{x,t}\sup_{r \in [t, t+h]} \Big| \mathcal{A}g^{K}(X_{r},r)+Ku_{r}\big(X_{r}-g^{K}(X_{r},r)\big) \Big| < \infty.
\end{equation*}
Then $\hat{u}_{r}=\hat{u}(X_{r},r),\forall r\in[t,+\infty)$ is an equilibrium control for (\ref{object_penal}), and $V^{K}$, $(f^{K})^{z,s}$ and $g^{K}$ are respectively the equilibrium value function and the auxiliary value functions given by
\begin{align}
&V^{K}(x,t)=I^{K}(x,t;\hat{u}),\label{intv}\\
&(f^{K})^{z,s}(x,t)=\mathbb{E}_{x,t}\int_{t}^{+\infty}K\hat{u}_{r}e^{-K\int_{t}^{r}\hat{u}_{r'}dr'}F(X_{r},r,z,s)dr,\label{intf}\\
&g^{K}(x,t)=\mathbb{E}_{x,t}\int_{t}^{+\infty}K\hat{u}_{r}e^{-K\int_{t}^{r}\hat{u}_{r'}dr'}X_{r}dr.\label{intg}
\end{align}
\end{theorem}
\begin{proof}
{\bf step 1:} We prove that $f^{K}$ and $g^{K}$ has the probability interpretation (\ref{intf})$\sim$(\ref{intg}).

For notational simplicity, define $R^{u}_{r}=e^{-K\int_{t}^{r}u_{r'}dr'}, \forall r\ge t$, then $d R^{u}_{r}=-KR^{u}_{r}u_{r}dr$. Define the stopping time $\tau_{n}:=\inf\{r\ge t||X_{r}|\ge n\}$ for all $n\ge 1$. Using It\^{o}-Krylov formula and (\ref{penf1}) yields that
\begin{align*}
\mathbb{E}_{x,t}R^{\hat{u}}_{T\wedge\tau_{n}}(f^{K})^{z,s}(X_{T\wedge\tau_{n}},T\wedge\tau_{n})-(f^{K})^{z,s}(x,t)=&\mathbb{E}_{x,t}\int_{t}^{T\wedge\tau_{n}}R^{\hat{u}}_{r}\Big[-K\hat{u}_{r}(f^{K})^{z,s}(X_{r},r)+\mathcal{A}(f^{K})^{z,s}(X_{r},r)\Big]dr\\
=&-\mathbb{E}_{x,t}\int_{t}^{T\wedge\tau_{n}}K\hat{u}_{r}R^{\hat{u}}_{r}F(X_{r},r,z,s)dr,
\end{align*}
where the It\^{o} integral term vanishes as $(f^{K})^{z,s}\in C^{1+\alpha,\frac{1+\alpha}{2}}(Q)$ implies the continuity of $(f^{K})^{z,s}_{x}$. Assumption \ref{ass_basic} ensures that $\lim\limits_{n\rightarrow\infty}\tau_{n}=+\infty$ almost surely.
Thanks to conditions (c) and (d), taking $n\rightarrow\infty$ and $T\rightarrow+\infty$ and using the dominated convergence theorem, we obtain (\ref{intf}). The proof of (\ref{intg}) is similar and omitted. For notational simplicity but less rigorously, we just write $\tau_{n}$ rather than $T\wedge \tau_{n}$, and only take $n\rightarrow\infty$ in later proofs.

{\bf step 2:} We prove (\ref{intv}). Using the results of the previous step, we only need to show that
\begin{equation}
\label{step2eq}
V^{K}(x,t)=f^{K}(x,t,x,t)+(G\diamond g^{K})(x,t).
\end{equation}

Let $(s,z)=(t,x)$ in (\ref{penf1}) and plug into (\ref{penv1}) with (\ref{peng1}), and we obtain
\begin{equation}
\label{eq1step2}
\begin{aligned}
&\mathcal{A}V^{K}(x,t)-\mathcal{A}f^{K}(x,t,x,t)+G_{y}(g^{K}(x,t),x,t)[K\hat{u}g^{K}(x,t)-K\hat{u}x]-\mathcal{A}(G\diamond g^{K})(x,t)\\&+K\hat{u}[G(g^{K}(x,t),x,t)-G_{y}(g^{K}(x,t),x,t)(g^{K}(x,t)-x)]-K\hat{u}V^{K}(x,t)+K\hat{u}f^{K}(x,t,x,t)=0.
\end{aligned}
\end{equation}
Applying It\^{o}-Krylov formula, it holds that
\begin{align*}
&\mathbb{E}_{x,t}R^{\hat{u}}_{\tau_{n}}V^{K}(X_{\tau_{n}},\tau_{n})-V^{K}(x,t)=\mathbb{E}_{x,t}\int_{t}^{\tau_{n}}\Big[-K\hat{u}_{r}V^{K}(X_{r},r)+\mathcal{A}V^{K}(X_{r},r)\Big]R^{\hat{u}}_{r}dr,\\
&\mathbb{E}_{x,t}R^{\hat{u}}_{\tau_{n}}f^{K}(X_{\tau_{n}},\tau_{n},X_{\tau_{n}},\tau_{n})-f^{K}(x,t,x,t)=\mathbb{E}_{x,t}\int_{t}^{\tau_{n}}\Big[-K\hat{u}_{r}f^{K}(X_{r},r,X_{r},r)+\mathcal{A}f^{K}(X_{r},r,X_{r},r)\Big]R^{\hat{u}}_{r}dr.\\
&\mathbb{E}_{x,t}R^{\hat{u}}_{\tau_{n}}(G\diamond g^{K})(X_{\tau_{n}},\tau_{n})-(G\diamond g^{K})(x,t)=\mathbb{E}_{x,t}\int_{t}^{\tau_{n}}\Big[-K\hat{u}_{r}(G\diamond g^{K})(X_{r},r)+\mathcal{A}(G\diamond g^{K})(X_{r},r)\Big]R^{\hat{u}}_{r}dr.
\end{align*}
Then, combining the last three formulas with (\ref{eq1step2}), we have
\begin{align*}
&[V^{K}(x,t)-f^{K}(x,t,x,t)-G(g^{K}(x,t),x,t)]\\
=&\mathbb{E}_{x,t}\int_{t}^{\tau_{n}}\Big[K\hat{u}_{r}V^{K}(X_{r},r)-\mathcal{A}V^{K}(X_{r},r)-K\hat{u}_{r}f^{K}(X_{r},r,X_{r},r)+\mathcal{A}f^{K}(X_{r},r,X_{r},r)\\&-K\hat{u}_{r}G(g^{K}(X_{r},r),X_{r},r)+\mathcal{A}(G\diamond g^{K})(X_{r},r)\Big]R^{\hat{u}}_{r}dr\\
&+\mathbb{E}_{x,t}R^{\hat{u}}_{\tau_{n}}[V^{K}(X_{\tau_{n}},\tau_{n})-f^{K}(X_{\tau_{n}},\tau_{n},X_{\tau_{n}},\tau_{n})-(G\diamond g^{K})(X_{\tau_{n}},\tau_{n})]\\
=&\mathbb{E}_{x,t}R^{\hat{u}}_{\tau_{n}}[V^{K}(X_{\tau_{n}},\tau_{n})-f^{K}(X_{\tau_{n}},\tau_{n},X_{\tau_{n}},\tau_{n})-(G\diamond g^{K})(X_{\tau_{n}},\tau_{n})].
\end{align*}
Thanks to condition (c), taking $n\rightarrow\infty$ in the last formula yields (\ref{step2eq})

{\bf step 3:} We prove that $\hat{u}$ is an equilibrium control policy. Recall that the perturbed control $u^{h}$ is defined by
\begin{equation*}
u^{h}_{r}=\begin{cases}
u_{r}, & r\in[t,t+h),\\
\hat{u}(r,X_{r}), & r\in[t+h,+\infty)
\end{cases}
\end{equation*}
with $u\in\mathcal{U}_{[t,+\infty)}$ being an arbitrary admissible control. Define 
\begin{align*}
&f^{u^{h}}(t,x,s,z):=\mathbb{E}_{x,t}\bigg\{e^{-K\int_{t}^{T}u^{h}_{r'}dr'}F(X_{T},T,z,s)+\int_{t}^{T}Ku^{h}_{r}e^{-K\int_{t}^{r}u^{h}_{r'}dr'}F(X_{r},r,z,s)dr\bigg\},\\
&g^{u^{h}}(x,t):=\mathbb{E}_{x,t}\bigg\{e^{-K\int_{t}^{T}u^{h}_{r'}dr'}X_{T}+\int_{t}^{T}Ku^{h}_{r}e^{-K\int_{t}^{r}u^{h}_{r'}dr'}X_{r}dr\bigg\}.
\end{align*}
It then holds that $f^{u^{h}}(X_{t+h},t+h,x,t)=f^{K}(X_{t+h},t+h,x,t)$ and $
g^{u^{h}}(X_{t+h},t+h)=g^{K}(X_{t+h},t+h)$. Direct calculation yields the fact that
\begin{equation}
\label{eq1step3}
\begin{aligned}
&I^{K}(x,t;u^{h})-I^{K}(x,t;\hat{u})\\
&=\mathbb{E}_{x,t}\bigg\{\int_{t}^{t+h}Ke^{-K\int_{t}^{r}u_{r'}dr'}F(X_{r},r,x,t)u_{r}dr+e^{-K\int_{t}^{t+h}u_{r}dr}f^{K}(X_{t+h},t+h,x,t)-f^{K}(x,t,x,t)\bigg\}\\
&\quad +G\bigg(\mathbb{E}_{x,t}\bigg\{\int_{t}^{t+h}Ke^{-K\int_{t}^{r}u_{r'}dr'}X_{r}u_{r}dr+e^{-K\int_{t}^{t+h}u_{r}dr}g^{K}(X_{t+h},t+h)\bigg\},x,t\bigg)-G(g^{K}(x,t),x,t).
\end{aligned}
\end{equation}
Applying It\^{o}-Krylov formula to $R^{u}_{r}(f^{K})^{x,t}(X_{r},r)$ and $R^{u}_{r}g^{K}(X_{r},r)$ from $t$ to $t+h$, we have
\begin{align*}
&\mathbb{E}_{x,t}\big[R^{u}_{t+h}(f^{K})^{x,t}(X_{t+h},t+h)\big]-(f^{K})^{x,t}(x,t)=\mathbb{E}_{x,t}\int_{t}^{t+h}R^{u}_{r}\big[-Ku_{r}(f^{K})^{x,t}(X_{r},r)+\mathcal{A}(f^{K})^{x,t}(X_{r},r)\big]dr,\\
&\mathbb{E}_{x,t}\big[R^{u}_{t+h}g^{K}(X_{t+h},t+h)\big]-g^{K}(x,t)=\mathbb{E}_{x,t}\int_{t}^{t+h}R^{u}_{r}\big[-Ku_{r}g^{K}(X_{r},r)+\mathcal{A}g^{K}(X_{r},r)\big]dr.
\end{align*}
Substitute the last two formula back into (\ref{eq1step3}) and utilize the Taylor expansion of $G$ around $g^{K}(x,t)$ as $h\downarrow 0$ with condition (e), we obtain that
\begin{align*}
&I^{K}(x,t;u^{h})-I^{K}(x,t;\hat{u})\\
=&\mathbb{E}_{x,t}\int_{t}^{t+h}R^{u}_{r}\Big[\mathcal{A}(f^{K})^{x,t}(X_{r},r)+K u_{r}\big(F(X_{r},r,x,t)-f^{K}(X_{r},r,x,t)\big)\Big]dr\\
&+G_{y}(g^{K}(x,t),x,t)\mathbb{E}_{x,t}\int_{t}^{t+h}R^{u}_{r}\Big[\mathcal{A}g^{K}(X_{r},r)+Ku_{r}\big(X_{r}-g^{K}(X_{r},r)\big)\Big]dr+o(h).
\end{align*}
Dividing by $h$ and letting $h\downarrow 0$, we obtain from Fatou's lemma that
\begin{equation}
\label{eq2step3}
\begin{aligned}
&\limsup_{h\downarrow 0}\frac{I^{K}(x,t;u^{h})-I^{K}(x,t;\hat{u})}{h}\\
&\le\mathcal{A}(f^{K})^{x,t}(x,t)+Ku_{t}\big(F(x,t,x,t)-f^{K}(x,t,x,t)\big)+G_{y}(g^{K}(x,t),x,t)\Big[\mathcal{A}g^{K}(x,t)+Ku_{t}\big(x-g^{K}(x,t)\big)\Big]\\
&=\mathcal{A}(f^{K})^{x,t}(x,t)+(\mathcal{H}g^{K})(x,t)+Ku_{t}\Big[F(x,t,x,t)-f^{K}(x,t,x,t)-G_{y}(g^{K}(x,t),x,t)\big(g^{K}(x,t)-x\big)\Big].
\end{aligned}
\end{equation}

From (\ref{step2eq}), we have
\begin{equation*}
\mathcal{A}V^{K}(x,t)=\mathcal{A}f^{K}(x,t,x,t)+\mathcal{A}(G\diamond g^{K})(x,t).
\end{equation*}
Substituting (\ref{step2eq}), (\ref{penv1}) and the last formula into (\ref{eq2step3}), we obtain
\begin{align*}
&\limsup_{h\downarrow 0}\frac{I^{K}(x,t;u^{h})-I^{K}(x,t;\hat{u})}{h}\\
\le &\mathcal{A}V^{K}(x,t)-\mathcal{A}f^{K}(x,t,x,t)+\mathcal{A}(f^{K})^{x,t}(x,t)+(\mathcal{H}g^{K})(x,t)-\mathcal{A}(G\diamond g^{K})(x,t)\\&+Ku_{t}\big[F(x,t,x,t)+[G(g^{K}(x,t),x,t)-G_{y}(g^{K}(x,t),x,t)(g^{K}(x,t)-x)]-V^{K}(x,t)\big]\\
\le&0.
\end{align*}
\end{proof}

Nevertheless, the solution to the candidate limit equations of (\ref{penv1})$\sim$(\ref{peng1}) also solves the stopping problem (\ref{object_stopping}) when the non-linear term is concave, i.e., $G_{yy}\le 0$. This property is presented in the following lemma.
\begin{lemma}
\label{lem_relation}
If $G_{yy}\le 0$, and $(V,\{f^{z,s}\}_{(z,s)\in Q},g)$ solves
\begin{align*}
&
\begin{aligned}\max\{&\mathcal{A}V(x,t)-\mathcal{A}f(x,t,x,t)+\mathcal{A}f^{x,t}(x,t)+(\mathcal{H}g)(x,t)-\mathcal{A}(G\diamond g)(x,t),\\&F(x,t,x,t)+G(g(x,t),x,t)-G_{y}(g(x,t),x,t)(g(x,t)-x)-V(x,t)\}=0,
\end{aligned}\\
&\mathcal{A}f^{z,s}(x,t)=0,\forall (x,t)\in \tilde{W},\\
&f^{z,s}(x,t)=F(x,t,z,s),\forall (x,t)\notin \tilde{W},\\
&\mathcal{A}g(x,t)=0,\forall (x,t)\in \tilde{W},\\
&g(x,t)=x,\forall (x,t)\notin \tilde{W},
\end{align*}
with 
\begin{equation*}
\tilde{W}:=\{(x,t)\in Q|F(x,t,x,t)+G(g(x,t),x,t)-G_{y}(g(x,t),x,t)(g(x,t)-x)-V(x,t)<0\},
\end{equation*}
then $(V,\{f^{z,s}\}_{(z,s)\in Q},g)$ is also a solution to (\ref{vi}).
\end{lemma}
\begin{proof}
On one hand, for any $(x,t)\in \tilde{W}^{c}$, it holds that $g(x,t)=x$ and $F(x,t,x,t)+G(g(x,t),x,t)-G_{y}(g(x,t),x,t)(g(x,t)-x)-V(x,t)\ge 0$, then $F(x,t,x,t)+G(x,x,t)-V(x,t)\ge 0$, which implies $(x,t)\in W^{c}$. On the other hand, for any $(x,t)\in \tilde{W}$, it holds that $F(x,t,x,t)+G(g(x,t),x,t)-G_{y}(g(x,t),x,t)(g(x,t)-x)-V(x,t)<0$. Note that $G_{yy}\le 0$ implies that $G(g(x,t),x,t)-G_{y}(g(x,t),x,t)(g(x,t)-x)\ge G(x,x,t)$. Then $F(x,t,x,t)+G(x,x,t)-V(x,t)< 0$, which implies $(x,t)\in W$.

Combining two results, we deduce that $\tilde{W}=W$, and the conclusion follows.
\end{proof}

In the following, we mainly focus on the convergence of the penalized PDE problem (\ref{pen}) towards (\ref{vi}). Under the condition $G_{yy}\le 0$, the convergence of penalized extended HJB system (\ref{penv1})$\sim$(\ref{peng1}) towards (\ref{vi}) holds similarly thanks to Lemma \ref{lem_relation}.

 \section{Convergence Analysis of Penalization}
 \label{Sect_conv}

 This section gives the main results of this paper. Subsection \ref{con_sub} gives the convergence of two subproblems of (\ref{pen}). Subsection \ref{con_tog} combines the two subproblems to obtain the overall convergence and an existence result.

 In the following, $p$ is sufficiently large to satisfy the embedding of $W^{2,1}_{p}$ into $C^{1+\alpha,\frac{1+\alpha}{2}}$ for some $\alpha\in(0,1)$.

\subsection{Convergence of two subproblems}
\label{con_sub}

The penalized PDE problem (\ref{pen}) can be decoupled into two subproblems: The first formula of (\ref{pen}) constitutes the first subproblem, which is a problem of solving $V^{K}$ given $(\{(f^{K})^{z,s}\}_{(z,s)\in Q},g^{K})$; The other four formulas of (\ref{pen}) constitute the second subproblem, which is a problem of solving $(\{(f^{K})^{z,s}\}_{(z,s)\in Q},g^{K})$ given $V^{K}$.

The next theorem gives the convergence of the first subproblem.

\begin{theorem}
\label{conv_1}
Under Assumption \ref{ass_basic}, given $(\{(f^{K})^{z,s}\}_{(z,s)\in Q},g^{K})$ for any $K\ge 1$ satisfying:\\
(a) $(f^{K})^{z,s},g^{K}\in W^{2,1}_{p,loc}(Q)\bigcap C(Q)$ for any $(z,s)\in Q$ and $K\ge 1$.\\
(b) As $K\rightarrow\infty$, $(f^{K})^{z,s}$ and $g^{K}$ converge to some functions $f^{z,s}$ and $g$ weakly in $W^{2,1}_{p,loc}(Q)$.\\
(c) There exists $C$ and $M$ independent of $t,s,K$ such that, for $\varphi=(f^{K})^{z,s}_{s},(f^{K})^{z,s}_{z},(f^{K})^{z,s}_{xz},(f^{K})^{z,s}_{zz}$ it holds $|\varphi(x,t,z,s)|\le C(|x|^{M}+|z^{M}|+1)$; and for $\varphi=g^{K},g^{K}_{t},g^{K}_{x},g^{K}_{xx}$, it holds that $|\varphi(x,t)|\le C(|x|^{M}+1)$.\\
Suppose there exists $V^{K}$ as a $W^{2,1}_{p,loc}(Q)\bigcap C(Q)$ solution to
\begin{align*}
&\mathcal{A}V^{K}(x,t)-\mathcal{A}f^{K}(x,t,x,t)+\mathcal{A}(f^{K})^{x,t}(x,t)+(\mathcal{H}g^{K})(x,t)-\mathcal{A}(G\diamond g^{K})(x,t)\\&+K(F(x,t,x,t)+G(x,x,t)-V^{K}(x,t))^{+}=0.
\end{align*}
Assume there exists $T_{0}>0$ such that $|V^{K}(x,T)|\le C(|x|^{M}+1)$ for any $T\in[T_{0},+\infty)$ for some constants $C$ and $M$. Then $V^{K}$ has a subsequence $V^{K}_{n}$ that converges as $K_{n}\rightarrow\infty$ weakly in $W^{2,1}_{p,loc}(Q)$ to a function $V$, which satisfies almost everywhere in $Q$ that
\begin{equation}
\label{limit_V}
\begin{aligned}
\max\{&\mathcal{A}V(x,t)-\mathcal{A}f(x,t,x,t)+\mathcal{A}f^{x,t}(x,t)+(\mathcal{H}g)(x,t)-\mathcal{A}(G\diamond g)(x,t),\\
&F(x,t,x,t)+G(x,x,t)-V(x,t)\}=0.
\end{aligned}
\end{equation}
Moreover, if for $\tau^{K}:=\inf\{r\ge t|F(X_{r},r,X_{r},r)+G(X_{r},X_{r},r)-V^{K}(X_{r},r)\ge 0\}$, it holds that
\begin{equation}
\label{addition_cond}
\begin{aligned}
&\sup_{K \ge 1} \mathbb{E}_{x,t} \left[ \int_t^{\tau^K} (1 + \vert{}X_r\vert{}^M) dr \right] < \infty,\\
&\sup_{K \ge 1} \mathbb{E}_{x,t} \left[ 1 + \vert{}X_{\tau^{K}}\vert{}^M \right] < \infty,\\
&\lim_{T\rightarrow\infty} \sup_K \mathbb{E}_{x,t}\big[(1+\vert{}X_T\vert{}^M)\mathbf{1}_{\{\tau^K>T\}}\big]=0,
\end{aligned}
\end{equation}
then the sequence $V^{K}$ itself converges to $V$.
\end{theorem}
\begin{proof}
Define 
\begin{align*}
&S^{K}_{1}(x,t):=-\mathcal{A}f^{K}(x,t,x,t)+\mathcal{A}(f^{K})^{x,t}(x,t)+(\mathcal{H}g^{K})(x,t)-\mathcal{A}(G\diamond g^{K})(x,t),\\
&S_{1}(x,t):=-\mathcal{A}f(x,t,x,t)+\mathcal{A}f^{x,t}(x,t)+(\mathcal{H}g)(x,t)-\mathcal{A}(G\diamond g)(x,t),\\
&S_{2}(x,t):=F(x,t,x,t)+G(x,x,t).
\end{align*}
By the assumptions, we have that $S^{K}_{1}(x,t)\in L^{p}_{loc}(Q), S_{2}(x,t)\in C^{2,1}(Q)$, and 
\begin{equation}
\label{bound_SK}
|S^{K}_{1}(x,t)|+|S_{2}(x,t)|+|\mathcal{A}S_{2}(x,t)|\le C(|x|^{M}+1)
\end{equation}
for some constants $C$ and $M$. 

The proof is split into four main steps.

{\bf Step 1: the local uniform polynomial bound of $V^{K}$:}

Define $u^{K}(x,t):=1_{\{S_{2}(x,t)-V^{K}(x,t)\ge 0\}}$ and $u^{K}_{r}:=u^{K}(X_{r},r)$. Define $\tau_{n}:=\inf\{r\ge t||X_{r}|>n\}$. Fix $T>T_{0}$ sufficiently large. Applying It\^{o}-Krylov formula for $R^{u^{K}}_{r}V^{K}(X_{r},r)$ with $R^{u}_{r}:=e^{-K\int_{t}^{r}u_{r'}dr'}$, we have  that
\begin{align*}
\mathbb{E}_{x,t}R^{u^{K}}_{T\wedge\tau_{n}}V^{K}(X_{T\wedge\tau_{n}},T\wedge\tau_{n})-V^{K}(x,t)=&\mathbb{E}_{x,t}\int_{t}^{T\wedge\tau_{n}}R^{u^{K}}_{r}\Big[\mathcal{A}V^{K}(X_{r},r)-Ku^{K}V^{K}(X_{r},r)\Big]dr\\
=&\mathbb{E}_{x,t}\int_{t}^{T\wedge\tau_{n}}R^{u^{K}}_{r}\Big[-S^{K}_{1}(X_{r},r)-Ku^{K}_{r}S_{2}(X_{r},r)\Big]dr,
\end{align*}
where the It\^{o} integral term vanishes as the continuous function $V^{K}_{x}$ is bounded on arbitrary compact subset of $Q$. Under Assumption \ref{ass_basic}, it holds that $\tau_{n}\rightarrow+\infty$ almost surely. Thanks to the polynomial growth of $S^{K}_{1},S_{2}$ and $V^{K}(\cdot,T)$ for $T>T_{0}$, taking $n\rightarrow\infty$ and using the dominated convergence theorem, we obtain that
\begin{equation*}
\mathbb{E}_{x,t}R^{u^{K}}_{T}V^{K}(X_{T},T)-V^{K}(x,t)=\mathbb{E}_{x,t}\int_{t}^{T}R^{u^{K}}_{r}\Big[-S^{K}_{1}(X_{r},r)-Ku^{K}_{r}S_{2}(X_{r},r)\Big]dr.
\end{equation*}
Note that $|\int_{t}^{T}Ku^{K}_{r}R^{u^{K}}_{r}S_{2}(X_{r},r)dr|\le\sup\limits_{r\in[t,T]}|S_{2}(X_{r},r)|\int_{t}^{T}Ku^{K}_{r}R^{u^{K}}_{r}dr\le \sup\limits_{r\in[t,T]}|S_{2}(X_{r},r)|$, we deduce that $V^{K}$ has polynomial growth uniformly in $K$ on $\mathbb{R}\times[0,T]$ with $T\in[T_{0},+\infty)$:
\begin{align}
|V^{K}(x,t)|\le&\mathbb{E}_{x,t}\Big[T\sup\limits_{r\in[t,T]}|S^{K}_{1}(X_{r},r)|+\sup\limits_{r\in[t,T]}|S_{2}(X_{r},r)|+V^{K}(X_{T},T)\Big]\notag\\
\le& \tilde{C}^{(T)}\Big[\mathbb{E}_{x,t}\sup\limits_{r\in[t,T]}|X_{r}|^{M}+1\Big]:=\Psi^{(T)}(x,t),\forall (x,t)\in\mathbb{R}\times[0,T],\label{bound_VK}
\end{align}
for some constants $\tilde{C}^{(T)}$ and $M$. The function $\Psi^{(T)}(x,t)$ is of polynomial growth in $x$ and it is dependent of $T$ but independent of $K$.

{\bf Step 2: the uniform bound of penalty term:}

We next show the boundedness of $KU^{K}(x,t)$ where 
\begin{equation*}
U^{K}(x,t):=(S_{2}(x,t)-V^{K}(x,t))^{+}.
\end{equation*}
Define $W^{K}:=\{(x,t)|S_{2}(x,t)-V^{K}(x,t)<0\}$, then we have $KU^{K}(x,t)=0$ on $W^{K}$, and 
\begin{equation*}
\mathcal{A}V^{K}(x,t)+S^{K}_{1}(x,t)+KU^{K}(x,t)=0
\end{equation*}
on $(W^{K})^{c}$. Then, we have
\begin{equation}
\label{eq_wkc}
\mathcal{A}U^{K}(x,t)=\mathcal{A}S_{2}(x,t)+S^{K}_{1}(x,t)+KU^{K}(x,t),\quad\forall (x,t)\in (W^{K})^{c}.
\end{equation}
Define a function $\phi(x):=(x^{2}+1)^{q}$ with $q$ sufficiently large to dominate the polynomial growth of $U^{K}(x,t)=(S_{2}(x,t)-V^{K}(x,t))^{+}$ obtained in {\bf Step 1} and $\mathcal{A}S_{2}(x,t)+S^{K}_{1}(x,t)$. Thanks to the linear growth of $\mu(x,t)$ and $\sigma(x,t)$, there exists $C'>0$ such that
\begin{equation*}
\mathcal{A}\phi(x)\le C'\phi(x).
\end{equation*}

Let us define and consider the function $w^{K}(x,t):=\frac{U^{K}(x,t)}{\phi(x)}$. It holds that $w^{K}(x,t)=0$ on $\partial (W^{K})^{c}$, $\limsup\limits_{t\rightarrow+\infty}\frac{U^{K}(x,t)}{\phi(x)}\le C'$ for some constant $C'>0$, and that $w^{K}(x,t)\rightarrow0$ as $x\rightarrow\infty$. Therefore, to show that $w^{K}$ is bounded on $(W^{K})^{c}$, we only need to show boundedness in the interior. Suppose that $(x_{K},t_{K})$ is a maximum of $w^{K}$ on $(W^{K})^{c}$ in the interior of $(W^{K})^{c}$. Since
\begin{align*}
&0\ge w^{K}_{t}(x_{K},t_{K})=\frac{U^{K}_{t}\phi-U^{K}\phi_{t}}{\phi^{2}}\Big|_{(x_{K},t_{K})},\\
&0=w^{K}_{x}(x_{K},t_{K})=\frac{U^{K}_{x}\phi-U^{K}\phi_{x}}{\phi^{2}}\Big|_{(x_{K},t_{K})},\\
&0\ge w^{K}_{xx}(x_{K},t_{K})=\Big[\frac{U^{K}_{xx}\phi-U^{K}\phi_{xx}}{\phi^{2}}-2\frac{\phi_{x}}{\phi}\frac{U^{K}_{x}\phi-U^{K}\phi_{x}}{\phi^{2}}\Big]\Big|_{(x_{K},t_{K})},
\end{align*}
we have that
\begin{equation*}
0\ge \mathcal{A}w^{K}(x_{K},t_{K})=\Big[\frac{\mathcal{A}U^{K}}{\phi}-\frac{U^{K}\mathcal{A}\phi}{\phi^{2}}\Big]\Big|_{(x_{K},t_{K})}.
\end{equation*}
Then we obtain that
\begin{equation*}
\mathcal{A}U^{K}(x_{K},t_{K})\le \frac{U^{K}(x_{K},t_{K})}{\phi(x_{K})}\mathcal{A}\phi(x_{K})\le C'U^{K}(x_{K},t_{K}).
\end{equation*}
Substituting the last formula into (\ref{eq_wkc}), we deduce that
\begin{equation*}
\mathcal{A}S_{2}(x_{K},t_{K})+S^{K}_{1}(x_{K},t_{K})+KU^{K}(x_{K},t_{K})\le C'U^{K}(x_{K},t_{K}).
\end{equation*}
Then, for $K>C'$, we have
\begin{equation*}
(K-C')w^{K}(x_{K},t_{K})\le\frac{\mathcal{A}S_{2}(x_{K},t_{K})+S^{K}_{1}(x_{K},t_{K})}{\phi(x_{K})}\le\bar{C}
\end{equation*}
for some constant $\bar{C}$ independent of $K$. The last inequality is due to the fact that $\phi(x,t)$ dominates the polynomial growth of $\mathcal{A}S_{2}(x,t)+S^{K}_{1}(x,t)$. Since $(x_{K},t_{K})$ is the maximum of $w^{K}$ on $(W^{K})^{c}$, then it holds that $KU^{K}(x,t)\le\frac{K\bar{C}\phi(x,t)}{K-C'}\le\bar{C}'\phi(x)$ on $(W^{K})^{c}$ for some $\bar{C}'>0$, and we arrive at the conclusion that 
\begin{equation}
\label{bound_source}
0\le KU^{K}(x,t)\le\bar{C}'\phi(x),\quad \forall (x,t)\in Q.
\end{equation}

{\bf Step 3: subsequence convergence and limit equations:}

Using the results of previous two steps and the interior $L^{p}$ estimate, it holds that for any compact cylinders $\mathcal{S}\subset\mathcal{S}'\subset Q$,
\begin{align*}
\Vert V^{K}\Vert_{W^{2,1}_{p}(\mathcal{S})}\le C\Big[\Vert S^{K}_{1}\Vert_{L^{p}(\mathcal{S}')}+\Vert KU^{K}\Vert_{L^{p}(\mathcal{S}')}+\Vert V^{K}\Vert_{L^{p}(\mathcal{S}')}\Big]\le C''
\end{align*}
for some constant $C''$ independent of $K$. Then, $\{V^{K}\}$ is precompact in $W^{2,1}_{p}(\mathcal{S})$ for any $\mathcal{S}$. Take an arbitrary sequence of increasing compact cylinder sets $\mathcal{S}_{m}\uparrow Q$. Using the standard diagonal argument, there exists a subsequence $\{V^{K_{n}}\}$ converging weakly in each  $W^{2,1}_{p}(\mathcal{S}_{m})$ to some function $V\in W^{2,1}_{p,loc}(Q)$. Then $V$ is of polynomial growth in $x$. Define $W=\{(x,t)|S_{2}(x,t)-V(x,t)<0\}$. 

As $\Vert K_n U^{K_n}\Vert_{L^p(\mathcal{S})}$ is uniformly bounded in $K$, we have
\begin{equation*}
\Vert U^{K_n}\Vert_{L^p(\mathcal{S})} = \frac{1}{K_n} \Vert K_n U^{K_n}\Vert_{L^p(\mathcal{S})} \le \frac{C''}{K_n} \longrightarrow 0 \quad \text{as } n \to \infty.
\end{equation*}
This implies $U^{K_n} = (S_2 - V^{K_n})^+ \to (S_2 - V)^+$ in $L^p_{loc}(Q)$. By uniform convergence on compact sets, $(S_2(x,t) - V(x,t))^+ = 0$ pointwise on $Q$, i.e.,
\begin{equation}
\label{eq_1}
S_2(x,t) - V(x,t) \le 0, \quad \forall (x,t) \in Q.
\end{equation}

Note that $K_n U^{K_n}(x,t) \ge 0$ a.e., it holds that $\mathcal{A}V^{K_n}(x,t) + S_1^{K_n}(x,t) = -K_n U^{K_n}(x,t) \le 0$ a.e. For any non-negative test function $\varphi \in C_c^\infty(Q)$ ($\varphi \ge 0$), we have
\begin{equation*}
\int_Q \big(\mathcal{A}V^{K_n}(x,t) + S_1^{K_n}(x,t)\big)\varphi(x,t) dxdt \le 0.
\end{equation*}
In view that $\mathcal{A}V^{K_n} + S_1^{K_n} \rightharpoonup \mathcal{A}V + S_1$ weakly in $L^p_{loc}(Q)$, taking $n \to \infty$ gives
\begin{equation*}
\int_Q \big(\mathcal{A}V(x,t) + S_1(x,t)\big)\varphi(x,t) dxdt \le 0,
\end{equation*}
which implies
\begin{equation}
\label{eq_2}
\mathcal{A}V(x,t) + S_1(x,t) \le 0 \quad \text{a.e. in } Q.
\end{equation}

For any compact set $\mathcal{S} \subset W$, by continuity of $V - S_2$, there exists $\delta > 0$ such that $V(x,t) - S_2(x,t) \ge \delta > 0$ on $\mathcal{S}$. By uniform convergence of $V^{K_n} - S_2 \to V - S_2$ on $\mathcal{S}$, for $n$ sufficiently large we have
\begin{equation*}
V^{K_n}(x,t) - S_2(x,t) \ge \frac{\delta}{2} > 0, \quad \forall (x,t) \in \mathcal{S}.
\end{equation*}
Thus $U^{K_n}(x,t) = (S_2(x,t) - V^{K_n}(x,t))^+ = 0$ identically on $\mathcal{S}$ for all large $n$. Consequently,
\begin{equation*}
\mathcal{A}V^{K_n}(x,t) + S_1^{K_n}(x,t) = 0 \quad \text{a.e. on } \mathcal{S}.
\end{equation*}
Taking weak limit in $L^p(\mathcal{S})$ as $n \to \infty$ yields
\begin{equation}
\label{eq_3}
\mathcal{A}V(x,t) + S_1(x,t) = 0 \quad \text{a.e. on } W.
\end{equation}

Combining (\ref{eq_1}), (\ref{eq_2}) and (\ref{eq_3}), we obtain that the limit function $V\in W^{2,1}_{p,loc}(Q)$, with polynomial growth in $x$, satisfies (\ref{limit_V}) a.e.. 

{\bf Step 4: convergence of whole sequence:}

We only need to show that any converging subsequence of $V^{K}$ converges to the same limit under the condition (\ref{addition_cond}).

Let $V^{K_{n}}$ be an arbitrary converging subsequence to some $V$. According to {\bf Step 3}, $V$ is a $W^{2,1}_{p,loc}(Q)$ solution to (\ref{limit_V}). Define $\tau^{V}=\inf\{r\ge t|S_{2}(X_{r},r)-V(X_{r},r)\ge 0\}$. Define the following set of stopping times:
\begin{align*}
\mathcal{T}:=\Big\{\tau\Big|&\mathbb{E}_{x,t} \Big[ \int_t^{\tau} (1 + \vert{}X_r\vert{}^M) dr \Big] < \infty,\quad \mathbb{E}_{x,t} \Big[ 1 + \vert{}X_\tau\vert{}^M \Big] < \infty,\\
&\lim_{T\rightarrow\infty} \mathbb{E}_{x,t}\big[(1+\vert{}X_T\vert{}^M)\mathbf{1}_{\{\tau>T\}}\big]=0.\Big\}.
\end{align*}
Notice that $\tau^{K_{n}}$ converges to $\tau^{V}$ almost surely. Using condition (\ref{addition_cond}) and Fatou's lemma, it holds that $\tau^{V}\in\mathcal{T}$.

For any $\tau\in \mathcal{T}$, $\tau_{n}:=\inf\{r\ge t|\vert{}X_{r}\vert{}\ge n\}$ and $T>T_{0}$, using the It\^{o}-Krylov formula, we get that
\begin{equation*}
\mathbb{E}_{x,t}V(X_{T\wedge\tau\wedge\tau_{n}},T\wedge\tau\wedge\tau_{n})-V(x,t)=\mathbb{E}_{x,t}\int_{t}^{T\wedge\tau\wedge\tau_{n}}\mathcal{A}V(X_{r},r)dr.
\end{equation*}
Sending $n\rightarrow\infty$ and using (\ref{limit_V}), we deduce that
\begin{equation*}
\mathbb{E}_{x,t}S_{2}(X_{\tau},\tau)1_{\tau\le T}+\mathbb{E}_{x,t}V(X_{T},T)1_{\tau>T}-V(x,t)\le\mathbb{E}_{x,t}\int_{t}^{T\wedge\tau}-S_{1}(X_{r},r)dr,
\end{equation*}
where equality holds when $\tau=\tau^{V}$. As $|V^{K}(x,T)|\le C(|x|^{M}+1)$ for $T>T_{0}$, it also holds if $V^{K}$ is replaced by the limit $V$. By virtue of the polynomial growth of $S_{1},S_{2},V$ and the definition of $\mathcal{T}$, we can apply the dominated convergence theorem to obtain that
\begin{equation}
\label{present_V}
V(x,t)=\sup\limits_{\tau\in\mathcal{T}} \mathbb{E}_{x,t}\bigg[\int_{t}^{\tau}S_{1}(X_{r},r)dr+S_{2}(X_{\tau},\tau)\bigg],
\end{equation}
where the maximum is achieved with $\tau=\tau^{V}$.

Then any convergent subsequence of $V^{K}$ converges to the same $V$ uniquely given by (\ref{present_V}), and the proof is complete.
\end{proof}

\begin{remark}
In Theorem \ref{conv_1}, $V^{K}$ solves a second-order semi-linear parabolic PDE with the non-linear term $K(S^{K}_{2}(x,t)-V^{K}(x,t))$ Lipschitzs in $V^{K}$. The regularity of solutions to this type of PDE can be $W^{2,1}_{p,loc}$.
\end{remark}

The next theorem gives the convergence of the second subproblem.

\begin{theorem}
\label{conv_2}
Under Assumption \ref{ass_basic}, given $V^{K}$ for any $K\ge 1$ satisfying:\\
(a) $V^{K}$ converges to some function $V$ pointwise.\\
Suppose there exist $(f^{K})^{z,s}$ and $g^{K}$ as $W^{2,1}_{p,loc}(Q)\bigcap C(Q)$ solutions to 
\begin{align*}
&\mathcal{A}(f^{K})^{z,s}(x,t)+K(F(x,t,z,s)-(f^{K})^{z,s}(x,t))1_{(x,t)\notin W^{K}}=0,\\
&\mathcal{A}g^{K}(x,t)+K(x-g^{K}(x,t))1_{(x,t)\notin W^{K}}=0,
\end{align*}
where
\begin{equation*}
W^{K}:=\{(x,t)|F(x,t,x,t)+G(x,x,t)-V^{K}(x,t)<0\}.
\end{equation*}
Assume there exist $T^{z,s}_{0},T_{0},M$ such that $|(f^{K})^{z,s}(x,T)|\le C^{z,s}(|x|^{M}+1)$ for any $T\in[T^{z,s}_{0},+\infty)$ for some constant $C^{z,s}$, and $|g^{K}(x,T)|\le C(|x|^{M}+1)$ for any $T\in[T_{0},+\infty)$ for some constants $C$. Then $((f^{K})^{z,s}, g^{K})$ has a subsequence $((f^{K_{n}})^{z,s},g^{K_{n}})$ that converges as $K_{n}\rightarrow\infty$ weakly in $W^{2,1}_{p,loc}(Q)$ to some $(f^{z,s}, g)$, which satisfies almost everywhere in $Q$ that
\begin{equation}
\label{limit_fg}
\begin{aligned}
&\mathcal{A}f^{z,s}(x,t)=0,\forall (x,t)\in W,\\
&F(x,t,z,s)-f^{z,s}(x,t)=0,\forall (x,t)\in W^{c},\\
&\mathcal{A}g(x,t)=0,\forall (x,t)\in W,\\
&x-g(x,t)=0,\forall (x,t)\in W^{c},
\end{aligned}
\end{equation}
where
\begin{equation*}
W:=\{(x,t)|F(x,t,x,t)+G(x,x,t)-V(x,t)<0\}.
\end{equation*}
Moreover, if it holds for $\tau^{V}:=\inf\{r\ge t|F(X_{r},r,X_{r},r)+G(X_{r},X_{r},r)\ge V(X_{r},r)\}$ that
\begin{equation}
\label{addition_cond_V}
\begin{aligned}
&\mathbb{E}_{x,t}\int_{t}^{\tau^{V}}(1+|X_{r}|^{M})dr<\infty,\\
&\mathbb{E}_{x,t}\big[1+|X_{\tau^{V}}|^{M}\big]<\infty,\\
&\lim\limits_{T\rightarrow\infty}\mathbb{E}_{x,t}\big[(1+|X_{T}|^{M})1_{\{\tau^{V}>T\}}\big]=0,
\end{aligned}
\end{equation}
then $((f^{K})^{z,s},g^{K})$ converges to $(f^{z,s}, g)$.
\end{theorem}
\begin{proof}
The procedure of the proof is similar to that of Theorem \ref{conv_1}, which is split into four steps.

{\bf Step 1: the local uniform polynomial bound of $(f^{K})^{z,s}$ and $g^{K}$:}

Define $u^{K}(x,t):=1_{\{(x,t)\notin W^{K}\}}$ and $u^{K}_{r}=u^{K}(X_{r},r)$. Define $\tau_{n}:=\inf\{r\ge t||X_{r}|>n\}$. Fix $T>T_{0}$ sufficiently large. Applying It\^{o}-Krylov formula for $R^{u^{K}}_{r}(f^{K})^{z,s}(X_{r},r)$ and $R^{u^{K}}_{r}g^{K}(X_{r},r)$, we have 
\begin{align*}
&(f^{K})^{z,s}(x,t)=\mathbb{E}_{x,t}\bigg[\int_{t}^{T\wedge\tau_{n}}Ku^{K}_{r}e^{-K\int_{t}^{r}u^{K}_{r'}dr'}F(X_{r},r,z,s)dr+e^{-K\int_{t}^{T\wedge\tau_{n}}u^{K}_{r'}dr'}(f^{K})^{z,s}(X_{T\wedge\tau_{n}},T\wedge\tau_{n})\bigg],\\
&g^{K}(x,t)=\mathbb{E}_{x,t}\bigg[\int_{t}^{T\wedge\tau_{n}}Ku^{K}_{r}e^{-K\int_{t}^{r}u^{K}_{r'}dr'}X_{r}dr+e^{-K\int_{t}^{T\wedge\tau_{n}}u^{K}_{r'}dr'}g^{K}(X_{T\wedge\tau_{n}},T\wedge\tau_{n})\bigg].
\end{align*}
Under Assumption \ref{ass_basic}, it holds that $\tau_{n}\rightarrow\infty$ almost surely. Thanks to the polynomial growth of $F$, and $(f^{K})^{z,s}(\cdot,T)$ for $T>T^{z,s}_{0}$, and $g^{K}(\cdot,T)$ for $T>T_{0}$, we can take $n\rightarrow\infty$ and use the dominated convergence theorem to obtain 
\begin{align*}
&(f^{K})^{z,s}(x,t)=\mathbb{E}_{x,t}\bigg[\int_{t}^{T}Ku^{K}_{r}e^{-K\int_{t}^{r}u^{K}_{r'}dr'}F(X_{r},r,z,s)dr+e^{-K\int_{t}^{T}u^{K}_{r'}dr'}(f^{K})^{z,s}(X_{T},T)\bigg],\forall T>T^{z,s}_{0},\\
&g^{K}(x,t)=\mathbb{E}_{x,t}\bigg[\int_{t}^{T}Ku^{K}_{r}e^{-K\int_{t}^{r}u^{K}_{r'}dr'}X_{r}dr+e^{-K\int_{t}^{T}u^{K}_{r'}dr'}g^{K}(X_{T},T)\bigg],\forall T>T_{0}.
\end{align*}

From the above two formulas, we can derive the following uniform polynomial bounds of $(f^{K})^{z,s}$ and $g^{K}$:
\begin{equation}
\label{bound_fgK}
\begin{aligned}
&|(f^{K})^{z,s}(x,t)|\le\mathbb{E}_{x,t}\Big[\sup\limits_{r\in [t,T]}|F(X_{r},r,z,s)|+|(f^{K})^{z,s}(X_{T},T)|\Big]\le \tilde{C}_{z,s}\Big[\mathbb{E}_{x,t}\sup\limits_{r\in[t,T]}|X_{r}|^{M}+1\Big]\\
&:=\Psi_{f}^{z,s,(T)}(x,t),\forall T>T^{z,s}_{0},\forall(x,t)\in\mathbb{R}\times[0,T],\\
&|g^{K}(x,t)|\le\mathbb{E}_{x,t}\Big[\sup\limits_{r\in [t,T]}|X_{r}|+|g^{K}(X_{T},T)|\Big]\le \tilde{C}\Big[\mathbb{E}_{x,t}\sup\limits_{r\in[t,T]}|X_{r}|^{M}+1\Big]\\
&:=\Psi^{(T)}_{g}(x,t),\forall T>T_{0},\forall(x,t)\in\mathbb{R}\times[0,T],
\end{aligned}
\end{equation}
where $M$, $\tilde{C}_{z,s}$ and $\tilde{C}$ are constants independent of $K$. The functions $\Psi_{f}^{z,s,(T)}(x,t)$ and $\Psi^{(T)}_{g}(x,t)$ are of polynomial growth in $x$ and dependent of $T$ but independent of $K$.

{\bf Step 2: the uniform polynomial bounds of penalty terms:}

Then we show the boundedness of $K(U^{K}_{f})^{z,s}$ and $KU^{K}_{g}$ where
\begin{align*}
&(U^{K}_{f})^{z,s}(x,t):=(F(x,t,z,s)-(f^{K})^{z,s}(x,t))1_{(x,t)\notin W^{K}},\\
&U^{K}_{g}(x,t):=(x-g^{K}(x,t))1_{(x,t)\notin W^{K}}.
\end{align*}
From the equations of $(f^{K})^{z,s}$ and $g^{K}$, we deduce that
\begin{align*}
&\mathcal{A}(U^{K}_{f})^{z,s}(x,t)=\Big[\mathcal{A}F(x,t,z,s)+K(U^{K}_{f})^{z,s}(x,t)\Big]1_{(x,t)\notin W^{K}},\\
&\mathcal{A}U^{K}_{g}(x,t)=\Big[\mathcal{A}I^{x}(x,t)+KU^{K}_{g}(x,t)\Big]1_{(x,t)\notin W^{K}},
\end{align*}
where $I^{x}(x,t)$ denotes the function $(x,t)\mapsto x$.

Define a function $\phi(x)=(x^{2}+1)^{q}$ with $q$ sufficiently large to dominate the polynomial growth of $F(x,t,z,s)-(f^{K})^{z,s}(x,t)$, $x-g^{K}(x,t)$, $\mathcal{A}F(x,t,z,s)$ and $\mathcal{A}I^{x}(x,t)$. By the linear growth of $\mu$ and $\sigma$, there exists $C'>0$ such that $\mathcal{A}\phi(x)\le C'\phi(x)$ for some constant $C'>0$. 

Now let us define and consider
\begin{equation*}
(w^{K}_{f})^{z,s}:=\frac{(U^{K}_{f})^{z,s}}{\phi},\quad w^{K}_{g}:=\frac{U^{K}_{g}}{\phi}.
\end{equation*}
It holds for $\varphi=(w^{K}_{f})^{z,s},w^{K}_{g}$ that $\varphi=0$ on $W^{K}$ and $\partial (W^{K})^{c}$, $\limsup\limits_{t\rightarrow+\infty}|\frac{(U^{K}_{f})^{z,s}(x,t)}{\phi(x)}|$ and $\limsup\limits_{t\rightarrow+\infty}|\frac{U^{K}_{g}(x,t)}{\phi(x)}|$ are finite, and that $\varphi(x,t)\rightarrow 0$ as $x\rightarrow\infty$. Therefore, to show the boundedness of $(w^{K}_{f})^{z,s}$ and $w^{K}_{g}$ on $(W^{K})^{c}$, we only need to show the boundedness in the interior. Suppose on $(W^{K})^{c}$ that $(x^{g,+}_{K},t^{g,+}_{K})$ is a maximum of $w^{K}_{g}$ while $(x^{g,-}_{K},t^{g,-}_{K})$ is a minimum, then
\begin{align*}
&(w^{K}_{g})_{t}(x^{g,+}_{K},t^{g,+}_{K})\le 0, (w^{K}_{g})_{x}(x^{g,+}_{K},t^{g,+}_{K})= 0, (w^{K}_{g})_{xx}(x^{g,+}_{K},t^{g,+}_{K})\le 0,\\
&(w^{K}_{g})_{t}(x^{g,-}_{K},t^{g,-}_{K})\ge 0, (w^{K}_{g})_{x}(x^{g,-}_{K},t^{g,-}_{K})= 0, (w^{K}_{g})_{xx}(x^{g,-}_{K},t^{g,-}_{K})\ge 0,
\end{align*}
which leads to
\begin{align*}
&0\ge \mathcal{A}w^{K}_{g}(x^{g,+}_{K},t^{g,+}_{K})=\Big[\frac{\mathcal{A}U^{K}_{g}}{\phi}-\frac{U^{K}_{g}\mathcal{A}\phi}{\phi^{2}}\Big]\Big|_{(x^{g,+}_{K},t^{g,+}_{K})},\\
&0\le \mathcal{A}w^{K}_{g}(x^{g,-}_{K},t^{g,-}_{K})=\Big[\frac{\mathcal{A}U^{K}_{g}}{\phi}-\frac{U^{K}_{g}\mathcal{A}\phi}{\phi^{2}}\Big]\Big|_{(x^{g,-}_{K},t^{g,-}_{K})}.
\end{align*}
Noting that $U^{K}_{g}(x^{g,+}_{K},t^{g,+}_{K})\ge 0$ while $U^{K}_{g}(x^{g,-}_{K},t^{g,-}_{K})\le 0$, we then have
\begin{align*}
&\mathcal{A}I^{x}(x^{g,+}_{K},t^{g,+}_{K})+KU^{K}_{g}(x^{g,+}_{K},t^{g,+}_{K})=\mathcal{A}U^{K}_{g}(x^{g,+}_{K},t^{g,+}_{K})\le U^{K}_{g}(x^{g,+}_{K},t^{g,+}_{K})\frac{\mathcal{A}\phi(x^{g,+}_{K})}{\phi(x^{g,+}_{K})}\le C'U^{K}_{g}(x^{g,+}_{K},t^{g,+}_{K}),\\
&\mathcal{A}I^{x}(x^{g,-}_{K},t^{g,-}_{K})+KU^{K}_{g}(x^{g,-}_{K},t^{g,-}_{K})=\mathcal{A}U^{K}_{g}(x^{g,-}_{K},t^{g,-}_{K})\ge U^{K}_{g}(x^{g,-}_{K},t^{g,-}_{K})\frac{\mathcal{A}\phi(x^{g,-}_{K})}{\phi(x^{g,-}_{K})}\ge C'U^{K}_{g}(x^{g,-}_{K},t^{g,-}_{K}).
\end{align*}
Then, for $K>C'$,
\begin{align*}
&(K-C')w^{K}_{g}(x^{g,+}_{K},t^{g,+}_{K})\le \frac{-\mathcal{A}I^{x}(x^{g,+}_{K},t^{g,+}_{K})}{\phi(x^{g,+}_{K})}\le \bar{C},\\
&(K-C')w^{K}_{g}(x^{g,-}_{K},t^{g,-}_{K})\ge \frac{-\mathcal{A}I^{x}(x^{g,-}_{K},t^{g,-}_{K})}{\phi(x^{g,-}_{K})}\ge -\bar{C}
\end{align*}
for some constant $\bar{C}>$. Then, we have
\begin{equation*}
|KU^{K}_{g}(x,t)|\le \bar{C}\phi(x),\forall (x,t)\in Q.
\end{equation*}
In a similar fashion, we can deduce that there exists some $\bar{C}_{z,s}>0$ such that
\begin{equation*}
|K(U^{K}_{f})^{z,s}(x,t)|\le \bar{C}_{z,s}\phi(x),\forall (x,t)\in Q.
\end{equation*}

{\bf Step 3: subsequence convergence and limit equations:}

Similar to the proof of Theorem \ref{conv_1}, using the uniform boundedness in previous two steps, one can show that $((f^{K})^{z,s}, g^{K})$ has a subsequence $((f^{K_{n}})^{z,s}, g^{K_{n}})$ converging weakly in each $W^{2,1}_{p}(\mathcal{S}_{m})$ to some $(f^{z,s}, g)\in W^{2,1}_{p,loc}(Q) \times W^{2,1}_{p,loc}(Q)$.

For any $(x,t)\in W$, it holds that $(x,t)\in W^{K_{n
}}$ for $K_{n}$ sufficiently large. Taking $K_{n}\rightarrow\infty$, we obtain that $\mathcal{A}f^{z,s}(x,t)=\mathcal{A}g(x,t)=0$ a.e. for $(x,t)\in W$. For $(x,t)\in W^{c}$, by the boundedness of $K(U^{K_{n}}_{f})^{z,s}$ and $KU^{K_{n}}_{g}$ and taking $K_{n}\rightarrow+\infty$, we obtain that $f^{z,s}(x,t)=F(x,t,z,s)$ and $g(x,t)=x$ a.e. for $(x,t)\in W^{c}$. Therefore, $f^{z,s}$ and $g$ satisfy (\ref{limit_fg}) a.e. in $Q$.

{\bf Step 4: convergence of whole sequence}

Finally, we show that any converging subsequences of $(f^{K})^{z,s}$ and $g^{K}$ converge to the same limits.

Let $(f^K_{n})^{z,s}$ and $g^{K_{n}}$ be arbitrary converging subsequences to some $f^{z,s}$ and $g$. Then the proof of last step implies that $f^{z,s}$ and $g$ are $W^{2,1}_{p}(Q)$ solution to (\ref{limit_fg}). 

Using the It\^{o}-Krylov formula and the localization argument, we have that, for $T>T_{0}$,
\begin{align*}
&\mathbb{E}_{x,t}f^{z,s}(X_{T\wedge\tau^{V}},T\wedge\tau^{V})-f^{z,s}(x,t)=\mathbb{E}_{x,t}\int_{t}^{T\wedge\tau^{V}}\mathcal{A}f^{z,s}(X_{r},r)dr,\\
&\mathbb{E}_{x,t}g(X_{T\wedge\tau^{V}},T\wedge\tau^{V})-g(x,t)=\mathbb{E}_{x,t}\int_{t}^{T\wedge\tau^{V}}\mathcal{A}g(X_{r},r)dr,
\end{align*}
which yield
\begin{align*}
&f^{z,s}(x,t)=\mathbb{E}_{x,t}\Big[F(X_{T},T,z,s)1_{\tau^{V}\le T}+f^{z,s}(X_{T},T)1_{\tau^{V}>T}\Big],\\
&g(x,t)=\mathbb{E}_{x,t}\Big[X_{T}1_{\tau^{V}\le T}+g(X_{T},T)1_{\tau^{V}>T}\Big].
\end{align*}
By the polynomial growth of $F$ and the fact that $f^{z,s}(x,T)\le C_{z,s}(|x|^{M}+1)$ for $T>T_{0}^{z,s}$ and $g(x,T)\le C(|x|^{M}+1)$ for $T>T_{0}$, we take $T\rightarrow\infty$ and apply the dominated convergence to obtain that
\begin{equation}
\label{present_fg}
\begin{aligned}
&f^{z,s}(x,t)=\mathbb{E}_{x,t}F(X_{\tau^{V}},\tau^{V},z,s),\\
&g(x,t)=\mathbb{E}_{x,t}X_{\tau^{V}}.
\end{aligned}
\end{equation}

Then any converging subsequence of $((f^{K})^{z,s}, g^{K})$ converges to the same $(f^{z,s},g)$ uniquely given by (\ref{present_fg}), which completes the proof.
\end{proof}

\begin{remark}
Note that (\ref{limit_fg}) is sightly different from that in (\ref{vi}) as $W^{c}=\{(x,t)|F(x,t,x,t)+G(x,x,t)-V(x,t)\ge 0\}$. Putting (\ref{limit_fg}) and (\ref{limit_V}) together, we get $W^{c}=\{(x,t)|F(x,t,x,t)+G(x,x,t)-V(x,t)=0\}$ and the equations coincide with (\ref{vi}).
\end{remark}

\subsection{Convergence of the composite problem}
\label{con_tog}

This subsection establishes the convergence of $(V^{K},f^{K},g^{K})$ as the solution to the penalized problems. We impose the additional assumption in this subsection:
\begin{assumption}
\label{ass_exist}
Assume that for any $K$ sufficiently large, the PDE problem (\ref{pen}) has a solution $(V^{K},\{(f^{K})^{z,s}\}_{(z,s)\in Q},g^{K})$ with $V^{K},(f^{K})^{z,s},g^{K}\in W^{2,1}_{p,loc}(Q)\bigcap C(Q)$.
\end{assumption}

Sequentially applying Theorem \ref{conv_1} and Theorem \ref{conv_2} leads to our first main convergence theorem as below.
\begin{theorem}
\label{main_conv1}
Under Assumptions \ref{ass_basic} and \ref{ass_exist}, if two conditions hold that:\\
(a) $(f^{K})^{z,s}$ and $g^{K}$ converge to some functions $f^{z,s}$ and $g$ weakly in $W^{2,1}_{p,loc}(Q)$.\\
(b) There exist $C$ and $M$ independent of $t,s,K$ such that, for $\varphi=(f^{K})^{z,s}_{s},(f^{K})^{z,s}_{z},(f^{K})^{z,s}_{xz},(f^{K})^{z,s}_{zz}$ it holds that $|\varphi(x,t,z,s)|\le C(|x|^{M}+|z^{M}|+1)$; and for $\varphi=g^{K},g^{K}_{t},g^{K}_{x},g^{K}_{xx}$, it holds that $|\varphi(x,t)|\le C(|x|^{M}+1)$.\\
Then $(V^{K},(f^{K})^{z,s},g^{K})$ has a subsequence $(V^{K_{n}},(f^{K_{n}})^{z,s},g^{K_{n}})$ that converges in $W^{2,1}_{p,loc}(Q)$ to some $(V,f^{z,s},g)$ that solves (\ref{vi}) a.e..\\
Moreover, if\\
(c) For $\tau^{K}:=\inf\{r\ge t|F(X_{r},r,X_{r},r)+G(X_{r},X_{r},r)-V^{K}(X_{r},r)\ge 0\}$, the conditions (\ref{addition_cond}) holds.\\
Then $(V^{K},(f^{K})^{z,s},g^{K})$ converges in $W^{2,1}_{p,loc}(Q)$ to some $(V,f^{z,s},g)$.
\end{theorem}

Sequentially applying Theorem \ref{conv_2} and Theorem \ref{conv_1} leads to the following second main convergence theorem. 
\begin{theorem}
\label{main_conv2}
Under Assumptions \ref{ass_basic} and \ref{ass_exist}, if the following conditions hold that:\\
(a) $V^{K}$ converges pointwise to some function $V$.\\
(b) The same polynomial growth conditions as in (b) of Theorem \ref{main_conv1}.\\
Then $(V^{K},(f^{K})^{z,s},g^{K})$ has a subsequence $(V^{K_{n}},(f^{K_{n}})^{z,s},g^{K_{n}})$ that converges in $W^{2,1}_{p,loc}(Q)$ to some $(V,f^{z,s},g)$ that solves (\ref{vi}) a.e..\\
Moreover, if\\
(c) For $\tau^{V}:=\inf\{r\ge t|F(X_{r},r,X_{r},r)+G(X_{r},X_{r},r)-V(X_{r},r)\ge 0\}$, the conditions (\ref{addition_cond_V}) holds.\\
Then $(V^{K},(f^{K})^{z,s},g^{K})$ converges in $W^{2,1}_{p,loc}(Q)$ to some $(V,f^{z,s},g)$.
\end{theorem}

The conditions for Theorem \ref{main_conv2} are much easier to verify than those of Theorem \ref{main_conv1}. First, pointwise convergence is usually easier to verify than $W^{2,1}_{p}$ convergence. Second, the condition (c) in Theorem \ref{main_conv2} is imposed on a single stopping time, while the condition (c) in Theorem \ref{main_conv1} is on a family of stopping times.

Later in Section \ref{Sect_example}, we will apply Theorem \ref{main_conv2} to construct solutions to the extended HJB of time-inconsistent stopping problems.

\section{Discussions and Extensions}
\label{Sect_coro}

This section focuses on corollaries and extensions of the main results in the last section.

The following corollary links the existence of solution to the extended HJB for the stopping problem to the existence of solution to the penalized PDE.

\begin{corollary}
\label{existence_stopping}
Under Assumptions \ref{ass_basic} and \ref{ass_exist}, assume uniform polynomial growth conditions (b) of Theorem \ref{main_conv1} holds, and either one of the following boundedness holds:\\
(a) $\{(f^{K})^{z,s}\}$ and $\{g^{K}\}$ are locally bounded in $W^{2,1}_{p}(\mathcal{S})$ uniformly in $K$.\\
($\text{a}'$) $\{V^{K}\}$ is locally bounded uniformly in $K$ and equicontinuous.\\
Then the extended HJB variational inequality (\ref{vi}) for the time-inconsistent stopping problem has a $W^{2,1}_{p,loc}(Q)$ solution $(V,\{f^{z,s}\}_{(z,s)\in Q},g)$.
\end{corollary}
\begin{proof}
{\bf Proof under condition (a):}
Take arbitrary sequence of compact cylinders $\mathcal{S}_{m}\uparrow Q$. Using a standard diagonal argument, the condition (a) implies that there is a subsequence of $f^{K},g^{K}$ that converges in $W^{2,1}_{p,loc}$ to some functions $f,g$. Then we can apply Theorem \ref{main_conv1} to obtain the existence of solution to (\ref{vi}).

{\bf Proof under condition ($\text{a}'$):}
Use the Arzela-Ascoli theorem, there is a subsequence of $V^{K}$ converging locally uniformly, which implies pointwise convergence. Then we can apply Theorem \ref{main_conv2} to obtain the existence of a solution to (\ref{vi}).
\end{proof}

\begin{remark}
In the last theorem, if additionally the time-inconsistent stopping problem has a unique equilibrium, then $V^{K},f^{K},g^{K}$ are convergent, and the limit constitutes a solution to (\ref{vi}).
\end{remark}

The first extension is from the penalized PDE problem (\ref{pen}) to the penalized control problem (\ref{object_penal}). Using Lemma \ref{lem_relation} and similar proofs, we have the following theorem.
\begin{theorem}
Under Assumption \ref{ass_basic}, assume $G_{yy}\le 0$, and that for each $K\ge 1$, the penalized control problem (\ref{object_penal}) has an equilibrium solution $(V^{K},\{(f^{K})^{z,s}\}_{(z,s)\in Q},g^{K}\})$ with $V^{K},(f^{K})^{z,s},g^{K}\in W^{2,1}_{p,loc}(Q)\bigcap C(Q)$ satisfying (\ref{penv1})$\sim$(\ref{peng1}). Assume the same conditions (a)$\sim$(c) of Theorem \ref{main_conv1} or (a)$\sim$(c) of Theorem \ref{main_conv2}, then $(V^{K},(f^{K})^{z,s},g^{K})$ converges in $W^{2,1}_{p,loc}(Q)$ to some $(V,f^{z,s},g)$, and $(V,\{f^{z,s}\}_{(z,s)\in Q},g)$ solves (\ref{vi}) a.e..
\end{theorem}

Another extension is from infinite-time problems to finite-time problems, i.e., the stopping time is restricted to be no greater than some terminal time $T$. The main Theorems \ref{main_conv1} and \ref{main_conv2} have similar counterparts for finite time problems, and the proofs are similar. As our explicit examples in Section \ref{Sect_example} are all infinite-time, we omit the theorems for finite-time problems.

\section{Examples}
\label{Sect_example}
To illustrate the convergence theory, we revisit the two classic time-inconsistent stopping problems in \cite{bjork2021time}. 

The standard approach is to directly obtain the equilibrium solution using a guess-and-verify approach. Instead, we first obtain the solutions to the penalized PDEs, and apply Theorem \ref{main_conv2} to show that the limit is a solution to the stopping problem.

\subsection{Non-exponential discount}
The first example is an infinite-time stopping problem with non-exponential discount.

The decision maker intends to sell an asset at some time. The price of the asset evolves as
\begin{equation}
\label{gbm}
\begin{cases}
dX_{r}=\mu X_{r}dr+\sigma X_{r}dB_{r},\\
X_{t}=x.
\end{cases}
\end{equation}
The objective is to sell the asset at the fixed cost of $S$. The value is discounted by
\begin{equation*}
\beta(r)=\beta e^{-\delta r}+(1-\beta)e^{-(\delta+\lambda)r},\ \forall r\ge 0.
\end{equation*}

The stopping objective is
\begin{equation*}
V(x):=\sup\limits_{\tau} J(x;\tau)=\sup\limits_{\tau}\mathbb{E}_{x,t}[\beta(\tau-t)(X_{\tau}-S)].
\end{equation*}

In this example,
\begin{align*}
&\mu(x,t)=\mu x,\ \sigma(x,t)=\sigma x,\\
&F(x,t,z,s)=\beta(t-s)(x-S),\ G(y,z,s)=0,
\end{align*}
and the extended HJB variational inequality is written as
\begin{equation}
\label{vi_example1}
\begin{aligned}
&\max\{\mu x V'(x)+\frac{1}{2}\sigma^{2}x^{2}V''(x)-f_{s}(x,t,t),(x-S)-V(x)\}=0,\\
&f_{t}(x,t,s)+\mu x f_{x}(x,t,s)+\frac{1}{2}\sigma^{2}x^{2}f_{xx}(x,t,s)=0,\forall x\in W,\\
&f(x,t,s)=\beta(t-s)(x-S),\forall x\notin W,
\end{aligned}
\end{equation}
where $V(x,t)$ and $f^{z,s}(x,t)$ is now simply $V(x)$ and $f^{s}(x,t):=f(x,t,s)$.

To illustrate our convergence theory, we do not directly solve (\ref{vi_example1}). Instead, we first solve the  penalized PDE problem, and then use Theorem \ref{main_conv2} to obtain a solution to (\ref{vi_example1}).

Accordingly, the penalized PDE problem of (\ref{vi_example1}) is
\begin{align*}
&\mu x (V^{K})'(x)+\frac{1}{2}\sigma^{2}x^{2}(V^{K})''(x)-f^{K}_{s}(x,t,t)+K[(x-S)-V^{K}(x)]^{+}=0,\\
&f^{K}_{t}(x,t,s)+\mu x f^{K}_{x}(x,t,s)+\frac{1}{2}\sigma^{2}x^{2}f^{K}_{xx}(x,t,s)+K[\beta(t-s)(x-S)-f^{K}(x,t,s)]1_{\{(x-S)-V^{K}(x)\ge 0\}}=0.
\end{align*}
To solve the penalized PDE problem, we conjecture that $\{x|(x-S)-V^{K}(x)\ge 0\}=[x^{*}_{K},+\infty)$ and
\begin{equation}
\label{vf_example1}
\begin{aligned}
&V^{K}(x)=\beta w^{K}(x;\delta)+(1-\beta)w^{K}(x;\delta+\lambda),\\
&f^{K}(x,t,s)=\beta e^{-\delta(t-s)}w^{K}(x;\delta)+(1-\beta)e^{-(\delta+\lambda)(t-s)}w^{K}(x;\delta+\lambda),
\end{aligned}
\end{equation}
where $w^{K}(x;\delta),\forall\delta>0$ satisfies
\begin{equation*}
\mu xw^{K}_{x}(x;\delta)+\frac{1}{2}\sigma^{2}x^{2}w^{K}_{xx}(x;\delta)-\delta w^{K}(x;\delta)+K(x-S-w^{K}(x;\delta))1_{\{x\in[x^{*}_{K},+\infty)\}}=0.
\end{equation*}
In view that $w^{K}$ is bounded at $x=0+$ and is at most linear at $x=+\infty$, we obtain that
\begin{align*}
&w^{K}(x;\delta)=\begin{cases}
A^{(K)}_{1}x^{\gamma_{1}},&x\in(0,x^{*}_{K}),\\
B^{(K)}_{1}x^{\kappa_{1}}+\frac{K}{K+\delta-\mu}x-\frac{SK}{K+\delta},&x\in(x^{*}_{K},+\infty),
\end{cases}\\
&w^{K}(x;\delta+\lambda)=\begin{cases}
A^{(K)}_{2}x^{\gamma_{2}},&x\in(0,x^{*}_{K}),\\
B^{(K)}_{2}x^{\kappa_{2}}+\frac{K}{K+\delta+\lambda-\mu}x-\frac{SK}{K+\delta+\lambda},&x\in(x^{*}_{K},+\infty),
\end{cases}
\end{align*}
for some constants $A^{(K)}_{1},A^{(K)}_{2},B^{(K)}_{1},B^{(K)}_{2}$ to be determined and 
\begin{align*}
&\gamma_{1}=\frac{\sqrt{(\mu-\frac{\sigma^{2}}{2})^{2}+2\delta\sigma^{2}}-(\mu-\frac{\sigma^{2}}{2})}{\sigma^{2}}>0,\ \gamma_{2}=\frac{\sqrt{(\mu-\frac{\sigma^{2}}{2})^{2}+2(\delta+\lambda)\sigma^{2}}-(\mu-\frac{\sigma^{2}}{2})}{\sigma^{2}}>0,\\
&\kappa_{1}=\frac{-\sqrt{(\mu-\frac{\sigma^{2}}{2})^{2}+2(\delta+K)\sigma^{2}}-(\mu-\frac{\sigma^{2}}{2})}{\sigma^{2}}<0,\ \kappa_{2}=\frac{-\sqrt{(\mu-\frac{\sigma^{2}}{2})^{2}+2(\delta+\lambda+K)\sigma^{2}}-(\mu-\frac{\sigma^{2}}{2})}{\sigma^{2}}<0.
\end{align*}
The penalized $w^{K}$ is expected to be in $W^{2}_{p,loc}\bigcap C^{1}$, and hence the constants $A^{(K)}_{1},B^{(K)}_{1},A^{(K)}_{2},B^{(K)}_{2},x^{*}_{K}$ satisfy the smooth-fit conditions:
\begin{align*}
&A^{(K)}_{1}(x^{*}_{K})^{\gamma_{1}}=B^{(K)}_{1}(x^{*}_{K})^{\kappa_{1}}+\frac{K}{K+\delta-\mu}x^{*}_{K}-\frac{SK}{K+\delta},\\
&A^{(K)}_{2}(x^{*}_{K})^{\gamma_{2}}=B^{(K)}_{2}(x^{*}_{K})^{\kappa_{2}}+\frac{K}{K+\delta+\lambda-\mu}x^{*}_{K}-\frac{SK}{K+\delta+\lambda},\\
&A^{(K)}_{1}\gamma_{1}(x^{*}_{K})^{\gamma_{1}-1}=B^{(K)}_{1}\kappa_{1}(x^{*}_{K})^{\kappa_{1}-1}+\frac{K}{K+\delta-\mu},\\
&A^{(K)}_{2}\gamma_{2}(x^{*}_{K})^{\gamma_{2}-1}=B^{(K)}_{2}\kappa_{2}(x^{*}_{K})^{\kappa_{2}-1}+\frac{K}{K+\delta+\lambda-\mu},\\
&\beta A^{(K)}_{1}(x^{*}_{K})^{\gamma_{1}}+(1-\beta)A^{(K)}_{2}(x^{*}_{K})^{\gamma_{2}}=x^{*}_{K}-S.
\end{align*}
It then follows that
\begin{align*}
&x^{*}_{K}=S\frac{\beta\frac{K}{K+\delta}\frac{\kappa_{1}}{\kappa_{1}-\gamma_{1}}+(1-\beta)\frac{K}{K+\delta+\lambda}\frac{\kappa_{2}}{\kappa_{2}-\gamma_{2}}-1}{\beta\frac{K}{K+\delta-\mu}\frac{\kappa_{1}-1}{\kappa_{1}-\gamma_{1}}+(1-\beta)\frac{K}{K+\delta+\lambda-\mu}\frac{\kappa_{2}-1}{\kappa_{2}-\gamma_{2}}-1},\\
&A^{(K)}_{1}=\frac{\frac{K}{K+\delta-\mu}(\kappa_{1}-1)x^{*}_{K}-S\frac{K}{K+\delta}\kappa_{1}}{(\kappa_{1}-\gamma_{1})(x^{*}_{K})^{\gamma_{1}}},\\
&A^{(K)}_{2}=\frac{\frac{K}{K+\delta+\lambda-\mu}(\kappa_{2}-1)x^{*}_{K}-S\frac{K}{K+\delta+\lambda}\kappa_{2}}{(\kappa_{2}-\gamma_{2})(x^{*}_{K})^{\gamma_{2}}},\\
&B^{(K)}_{1}=\frac{\frac{K}{K+\delta-\mu}(\gamma_{1}-1)x^{*}_{K}-S\frac{K}{K+\delta}\gamma_{1}}{(\kappa_{1}-\gamma_{1})(x^{*}_{K})^{\kappa_{1}}},\\
&B^{(K)}_{2}=\frac{\frac{K}{K+\delta+\lambda-\mu}(\gamma_{2}-1)x^{*}_{K}-S\frac{K}{K+\delta+\lambda}\gamma_{2}}{(\kappa_{2}-\gamma_{2})(x^{*}_{K})^{\kappa_{2}}}.
\end{align*}

\begin{proposition}
For $V^{K}$ and $(f^{K})^{s}$ given in (\ref{vf_example1}), the pointwise limits 
\begin{align*}
&V(x)=\begin{cases}
\Big[\beta(\frac{x}{x^{*}})^{\gamma_{1}}+(1-\beta)(\frac{x}{x^{*}})^{\gamma_{2}}\Big](x^{*}-S),&x\in(0,x^{*})\\
x-S,&x\in[x^{*},+\infty).
\end{cases},\\
&f(x,t,s)=\begin{cases}
\Big[\beta e^{-\delta(t-s)}(\frac{x}{x^{*}})^{\gamma_{1}}+(1-\beta)e^{-(\delta+\lambda)(t-s)}(\frac{x}{x^{*}})^{\gamma_{2}}\Big](x^{*}-S),&x\in(0,x^{*})\\
\beta(t-s)(x-S),&x\in[x^{*},+\infty),
\end{cases}
\end{align*}
with
\begin{equation*}
x^{*}=\frac{S}{1-\frac{1}{\beta\gamma_{1}+(1-\beta)\gamma_{2}}}
\end{equation*}
is a $W^{2,1}_{p}(Q)\bigcap C(Q)$ solution to (\ref{vi_example1}).
\end{proposition}
\begin{proof}
It is straightforward to verify that
\begin{align*}
&\lim\limits_{K\to\infty}\kappa_{1}=\lim\limits_{K\to\infty}\kappa_{2}=-\infty,\\ &\lim\limits_{K\to\infty}x^{*}_{K}=\lim\limits_{K\to\infty}S\frac{\beta\frac{\gamma_{1}}{\kappa_{1}}+(1-\beta)\frac{\gamma_{2}}{\kappa_{2}}+O(\frac{1}{K})}{\beta\frac{\gamma_{1}-1}{\kappa_{1}}+(1-\beta)\frac{\gamma_{2}-1}{\kappa_{2}}+O(\frac{1}{K})}=x^{*},\\
&\lim\limits_{K\to\infty}A^{(K)}_{1}=\frac{x^{*}-S}{(x^{*})^{\gamma_{1}}},\quad\lim\limits_{K\to\infty}A^{(K)}_{2}=\frac{x^{*}-S}{(x^{*})^{\gamma_{2}}},\\
&\lim\limits_{K\to\infty}B^{(K)}_{1}x^{\kappa_{1}}=\lim\limits_{K\to\infty}B^{(K)}_{2}x^{\kappa_{2}}=0,\quad \forall x\in(x^{*},+\infty).
\end{align*}
Then it follows that $V^{K}\rightarrow V$ and $f^{K}\rightarrow f$ pontwise.

In this example, the condition (b) in Theorem \ref{main_conv2} only requires that $f^{K}_{s}(x,t,s)\le C(|x|^{M}+1),\forall s<t,\forall x$ for some constants $C$ and $M$ independent of $K$. By the expression of $f^{K}$, it is sufficient to show that the corresponding growth condition holds for $w^{K}(x;\delta)$ and $w^{K}(x;\delta+\lambda)$. 

Using the limits of $A^{(K)}_{1},A^{(K)}_{2}$ and $x^{*}_{K}$, we deduce that $w^{K}(x;\delta)$ and $w^{K}(x;\delta+\lambda)$ are uniformly bounded by some constants $C$ on $[0,x^{*}_{K}]$. For $x\in[x^{*}_{K},+\infty)$, the linear parts $\frac{K}{K+\delta-\mu}x-\frac{SK}{K+\delta}$ and $\frac{K}{K+\delta+\lambda-\mu}x-\frac{SK}{K+\delta+\lambda}$ are clearly uniformly bounded by polynomials. The term $B^{(K)}_{i}x^{\kappa_{i}},i=1,2$ on $[x^{*}_{K},+\infty)$ are actually uniformly bounded, because
\begin{align*}
&|B^{(K)}_{1}x^{\kappa_{1}}|\le |B^{(K)}_{1}(x^{*}_{K})^{\kappa_{1}}|= |A^{(K)}_{1}(x^{*}_{K})^{\gamma} - \frac{K}{K+\delta-\mu}x^{*}_{K} + \frac{SK}{K+\delta}|,\\
&|B^{(K)}_{2}x^{\kappa_{2}}|\le |B^{(K)}_{2}(x^{*}_{K})^{\kappa_{2}}|= |A^{(K)}_{2}(x^{*}_{K})^{\gamma} - \frac{K}{K+\delta+\lambda-\mu}x^{*}_{K} + \frac{SK}{K+\delta+\lambda}|,
\end{align*}
where the expression on the RHS is uniformly bounded.

It follows from Theorem \ref{main_conv2} that $(V^{K},f^{K})$ has a convergent subsequence, and the limit is a solution to (\ref{vi_example1}). As the pointwise limits of $(V^{K},f^{K})$ is $(V,f)$, the conclusion holds.
\end{proof}

\begin{remark}
It is worth noting that condition (c) of Theorem \ref{main_conv2} only holds for $\mu-\frac{\sigma^{2}}{2}>0$. It shows that condition (c) is only sufficient but not necessary for the convergence of the whole sequence.
\end{remark}

\begin{figure}[htbp]
\centering
\includegraphics[width=0.95\linewidth, height=1\textheight, keepaspectratio]{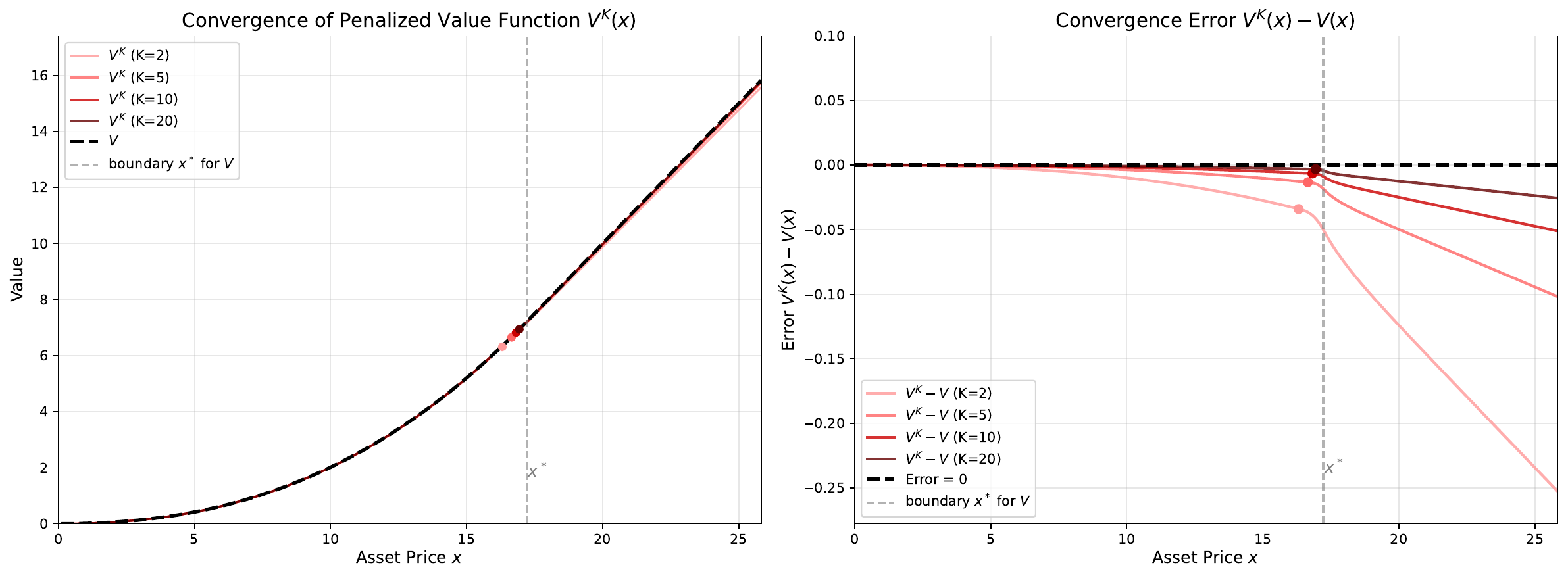}
\caption{Convergence in the non-exponential discount case. The model parameters are set as $\mu=0.02,\sigma=0.1,\delta=0.05,\lambda=0.03,\beta=0.5,S=10$.}
\label{fig1}
\end{figure}

As shown in Fig \ref{fig1}, the error will explode as $x$ tends to $+\infty$, which implies that $L^{\infty}$ convergence is not valid but the local convergence is appropriate in this case. We have the similar observation in the next mean-variance example in Fig \ref{fig2}.

\subsection{Mean-variance objective}
The second example is an 

The next example is still an asset selling problem, with the asset price following the geometric Brownian motion (\ref{gbm}). We consider an infinite-time stopping problem under a mean-variance objective, which is to maximize 
\begin{equation*}
V(x):=\sup\limits_{\tau} J(x;\tau)=\mathbb{E}_{x,t}X_{\tau}-\frac{\gamma}{2}\var_{x,t}X_{\tau}.
\end{equation*}

In this case, 
\begin{align*}
&\mu(x,t)=\mu x,\ \sigma(x,t)=\sigma x,\\
&F(x,t,z,s)=x-\frac{\gamma}{2}x^{2},\ G(y,z,s)=\frac{\gamma}{2}y^{2}.
\end{align*}
The extended HJB variational inequality is derived as
\begin{equation}
\label{vi_example2}
\begin{aligned}
&\max\{\mu xV'(x)+\frac{1}{2}\sigma^{2}x^{2}V''(x)-\frac{\gamma}{2}\sigma^{2}x^{2}[g'(x)]^{2},x-V(x)\}=0,\\
&\mu xg'(x)+\frac{1}{2}\sigma^{2}x^{2}g''(x)=0,\forall x\in W,\\
&g(x)=x,\forall x\notin W,
\end{aligned}
\end{equation}
where $V(x,t)$ and $g(x,t)$ are simply $V(x)$ and $g(x)$.

Same as before, we first consider and solve the penalized PDE problem
\begin{align*}
&\mu x(V^{K})'(x)+\frac{1}{2}\sigma^{2}x^{2}(V^{K})''(x)-\frac{\gamma}{2}\sigma^{2}x^{2}[(g^{K})'(x)]^{2}+K[x-V^{K}(x)]^{+}=0,\\
&\mu x(g^{K})'(x)+\frac{1}{2}\sigma^{2}x^{2}(g^{K})''(x)+K[x-g^{K}(x)]1_{\{x-V^{K}(x)\ge 0\}}=0.
\end{align*}
We conjecture that $\{x|x-V^{K}(x)\ge 0\}=[x^{*}_{K},+\infty)$. Using the boundedness at $x=0+$ and linear growth at $x=+\infty$, we obtain that
\begin{equation}
\label{vg_example2}
\begin{aligned}
&g^{K}(x)=\begin{cases}
A^{(K)}x^{\theta}, & x\in(0,x^{*}_{K}),\\
B^{(K)}x^{\kappa}+\frac{K}{K-\mu}x, & x\in(x^{*}_{K},+\infty),
\end{cases}\\
&V^{K}(x)=\begin{cases}
C^{(K)}x^{\theta}+\frac{\gamma}{2}[A^{(K)}]^{2}x^{2\theta}, & x\in(0,x^{*}_{K}),\\
D^{(K)}x^{\kappa} + \frac{K}{K-\mu}x + M^{(K)}_{1} x^{2\kappa} + M^{(K)}_{2} x^{\kappa+1} + M^{(K)}_{3} x^{2}, & x\in(x^{*}_{K},+\infty),
\end{cases}
\end{aligned}
\end{equation}
where
\begin{align*}
&\theta:=1-\frac{2\mu}{\sigma^{2}},\quad \kappa:=\frac{-\sqrt{(\mu-\frac{\sigma^{2}}{2})^{2}+2K\sigma^{2}}-(\mu-\frac{\sigma^{2}}{2})}{\sigma^{2}}<0,\\
&M^{(K)}_{1}:=\frac{\gamma\sigma^{2}\kappa^{2}[B^{(K)}]^{2}}{2[\kappa(\sigma^{2}-2\mu)+3K]},\ M^{(K)}_{2}:=\frac{\gamma\sigma^{2}\kappa K B^{(K)}}{(K-\mu)(\sigma^{2}\kappa+\mu)},\ M^{(K)}_{3}:=\frac{\gamma\sigma^{2}K^{2}}{2(K-\mu)^{2}(\sigma^{2}+2\mu-K)}.
\end{align*}
We note that for the boundedness at $x=0+$, it is required that $2\theta\ge 1$, i.e., $\mu\in(0,\frac{\sigma^{2}}{4}]$.

The constants $A^{(K)},B^{(K)},C^{(K)},D^{(K)}$ and $x^{*}_{K}$ satisfy the smooth-fit conditions:
\begin{align*}
&A^{(K)}(x^{*}_{K})^{\theta}=B^{(K)}(x^{*}_{K})^{\kappa}+\frac{K}{K-\mu}x^{*}_{K},\\
&A^{(K)}\theta(x^{*}_{K})^{\theta-1}=B^{(K)}\kappa(x^{*}_{K})^{\kappa-1}+\frac{K}{K-\mu},\\
&\begin{aligned}
C^{(K)}(x^{*}_{K})^{\theta}+\frac{\gamma}{2}[A^{(K)}]^{2}(x^{*}_{K})^{2\theta}=&D^{(K)}(x^{*}_{K})^{\kappa} + \frac{K}{K-\mu}x^{*}_{K} + M^{(K)}_{1} (x^{*}_{K})^{2\kappa} \\&+ M^{(K)}_{2} (x^{*}_{K})^{\kappa+1} +M^{(K)}_{3} (x^{*}_{K})^{2},
\end{aligned}\\
&\begin{aligned}
C^{(K)}\theta(x^{*}_{K})^{\theta-1}+\frac{\gamma}{2}[A^{(K)}]^{2}2\theta(x^{*}_{K})^{2\theta-1}=&D^{(K)}\kappa(x^{*}_{K})^{\kappa-1} + \frac{K}{K-\mu} + M^{(K)}_{1} 2\kappa(x^{*}_{K})^{2\kappa-1} \\&+ M^{(K)}_{2} (\kappa+1)(x^{*}_{K})^{\kappa} +M^{(K)}_{3} 2(x^{*}_{K}),
\end{aligned}\\
&x^{*}_{K}=C^{(K)}(x^{*}_{K})^{\theta}+\frac{\gamma}{2}[A^{(K)}]^{2}(x^{*}_{K})^{2\theta}.
\end{align*}
We deduce that
\begin{align*}
x^{*}_{K}=&\frac{\frac{2(K-\mu)}{\gamma K^{2}}[\mu(\kappa-\theta)-K(1-\theta)]}{\frac{\sigma^{2}\kappa^{3}}{\kappa(\sigma^{2}-2\mu)+3K}(\frac{\theta-1}{\kappa-\theta})^{2}+\frac{2\sigma^{2}\kappa}{\sigma^{2}\kappa+\mu}\frac{\theta-1}{\kappa-\theta}+(2-\kappa)\frac{\sigma^{2}}{\sigma^{2}+2\mu-K}-\theta(\frac{\kappa-1}{\kappa-\theta})^{2}},\\
A^{(K)}=&\frac{K}{K-\mu}\frac{\kappa-1}{\kappa-\theta}(x^{*}_{K})^{1-\theta},\\
B^{(K)}=&\frac{K}{K-\mu}\frac{\theta-1}{\kappa-\theta}(x^{*}_{K})^{1-\kappa},\\
C^{(K)}=&\frac{K}{K-\mu}\frac{\kappa-1}{\kappa-\theta}(x^{*}_{K})^{1-\theta}+M^{(K)}_{1}\frac{-\kappa}{\kappa-\theta}(x^{*}_{K})^{2\kappa-\theta}+M^{(K)}_{2}\frac{-1}{\kappa-\theta}(x^{*}_{K})^{\kappa-\theta+1}\\&+M^{(K)}_{3}\frac{\kappa-2}{\kappa-\theta}(x^{*}_{K})^{2-\theta}-\frac{\gamma}{2}[A^{(K)}]^{2}\frac{\kappa-2\theta}{\kappa-\theta}x^{\theta},\\
D^{(K)}=&\frac{K}{K-\mu}\frac{\theta-1}{\kappa-\theta}(x^{*}_{K})^{1-\kappa}+M^{(K)}_{1}\frac{\theta-2\kappa}{\kappa-\theta}(x^{*}_{K})^{\kappa}+M^{(K)}_{2}\frac{\theta-\kappa-1}{\kappa-\theta}x^{*}_{K}\\&+M^{(K)}_{3}\frac{\theta-2}{\kappa-\theta}(x^{*}_{K})^{2-\kappa}+\frac{\gamma}{2}[A^{(K)}]^{2}\frac{\theta}{\kappa-\theta}(x^{*}_{K})^{2\theta-\kappa}.
\end{align*}

\begin{proposition}
For $\mu\in (0,\frac{\sigma^{2}}{4}]$, and $V^{K}$ and $(f^{K})^{s}$ given in (\ref{vg_example2}), the pointwise limits 
\begin{align*}
&V(x)=\begin{cases}
[(x^{*})^{\frac{2\mu}{\sigma^{2}}}-\frac{\gamma}{2}(x^{*})^{1+\frac{2\mu}{\sigma^{2}}}]x^{1-\frac{2\mu}{\sigma^{2}}}+\frac{\gamma}{2}(x^{*})^{\frac{4\mu}{\sigma^{2}}}x^{2-\frac{4\mu}{\sigma^{2}}},&x\in(0,x^{*})\\
x,&x\in[x^{*},+\infty).
\end{cases},\\
&g(x)=\begin{cases}
(x^{*})^{\frac{2\mu}{\sigma^{2}}}x^{1-\frac{2\mu}{\sigma^{2}}},&x\in(0,x^{*})\\
x,&x\in[x^{*},+\infty),
\end{cases}
\end{align*}
with
\begin{equation*}
x^{*}=\frac{4\mu}{\gamma(\sigma^{2}-2\mu)}
\end{equation*}
is a $W^{2,1}_{p}(Q)\bigcap C(Q)$ solution to (\ref{vi_example2}).
\end{proposition}
\begin{proof}
As $K\to \infty$, it holds that $\kappa\sim-\sqrt{K}$, 
\begin{equation*}
\lim\limits_{K\to\infty}x^{*}_{K}=\frac{-\frac{2}{\gamma}(1-\theta)}{-\theta}=\frac{4\mu}{\gamma(\sigma^{2}-2\mu)}=x^{*},
\end{equation*}
and
\begin{align*}
&\lim\limits_{K\to\infty}A^{(K)}=(x^{*})^{\frac{2\mu}{\sigma^{2}}},\\
&\lim\limits_{K\to\infty}B^{(K)}x^{\kappa}=0,\forall x\in(x^{*},+\infty),\\
&\lim\limits_{K\to\infty}C^{(K)}=(x^{*})^{\frac{2\mu}{\sigma^{2}}}-\frac{\gamma}{2}(x^{*})^{1+\frac{2\mu}{\sigma^{2}}},\\
&\lim\limits_{K\to\infty}D^{(K)}x^{\kappa}=0,\forall x\in(x^{*},+\infty).
\end{align*}
Then it follows that $V^{K}\rightarrow V$ and $g^{K}\rightarrow g$ pointwise.

In this example, the polynomial growth condition (b) of Theorem \ref{main_conv2} requires that for $\varphi=g^{K},(g^{K})',(g^{K})''$, there exist constants $C$ and $M$ such that $|\varphi(x)|\le C(x^{M}+1)$.

Using the explicit expression of $g^{K}$, we only have that
\begin{equation*}
|g^{K}(x)|,|x(g^{K})'(x)|,|x^{2}(g^{K})''(x)|\le C(x^{M}+1)
\end{equation*}
for some constants $C$ and $M$, while the uniform polynomial growth condition (b) actually fails for $\varphi=(g^{K})',(g^{K})''$. However, the proof of Theorem \ref{conv_1}, Theorem \ref{conv_2} still holds  because we actually need the polynomial bound for $\mathcal{A}g^{K}(x)=\mu x(g^{K})'(x)+\frac{1}{2}\sigma^{2}x^{2}(g^{K})''$. Therefore, the subsequence convergence result of Theorem \ref{main_conv2} still holds in this example. Then, We conclude that $(V^{K},g^{K})$ has a convergent subsequence, and the limit is a solution to (\ref{vi_example2}). As the pointwise limits of $(V^{K},g^{K})$ is $(V,g)$, the conclusion holds.
\end{proof}

\begin{figure}[htbp]
\centering
\includegraphics[width=0.95\linewidth, height=1\textheight, keepaspectratio]{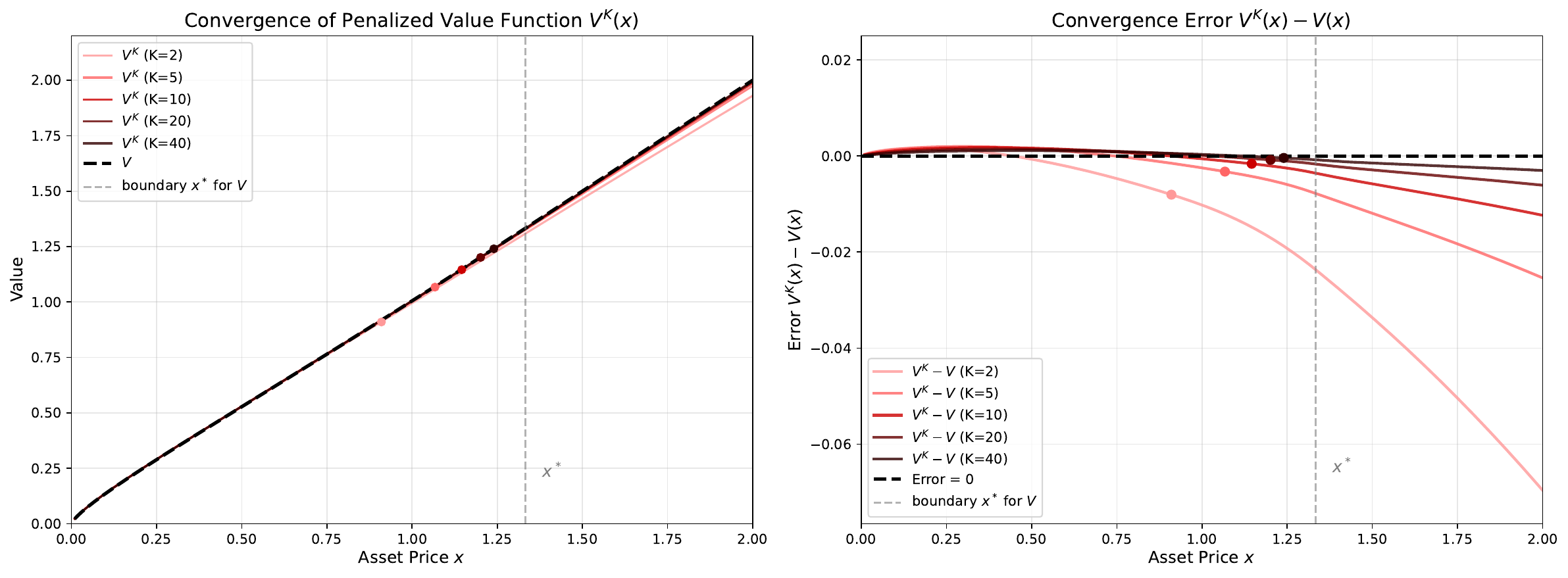}
\caption{Convergence in the mean-variance objective case. The model parameters are set as $\mu=0.02,\sigma=0.4,\gamma=0.5$.}
\label{fig2}
\end{figure}

\noindent
\textbf{Acknowledgements}:  X. Luo and X. Yu are supported by the Hong Kong RGC General Research Fund (GRF) under grant no. 15214125 and the Research Centre for Quantitative Finance at the Hong Kong Polytechnic University under grant no. P0042708.

\bibliographystyle{plainnat}
\bibliography{references}
\end{document}